\documentclass[a4paper,11pt,oneside]{article}

\usepackage{amsmath,amsthm,amsfonts,amssymb,amscd}
\usepackage{mathtools}
\usepackage{pgf,tikz}
\usepackage{mathrsfs}
\usepackage{float}
\usepackage{csvsimple}
\usepackage{enumerate}
\usepackage{booktabs}
\usepackage{pgfplots}
\usepackage[noend]{algpseudocode}
\usepackage{graphicx,bm,xcolor}
\usepackage{algorithm}
\usepackage{pdfpages}
\usepackage{algpseudocode}
\usepackage{stmaryrd}
\usepackage[square,sort,comma,numbers]{natbib}
\usetikzlibrary{arrows}
\usepackage{caption}

\usepackage{t1enc}
\usepackage[german,english]{babel}
\usepackage{verbatim} 

\usepackage{url}

\newif\ifMAKEPICS
\MAKEPICSfalse

\ifMAKEPICS
\usepackage[cleanup,subfolder]{gnuplottex}
\usepackage{xparse}

\ExplSyntaxOn
\DeclareExpandableDocumentCommand{\convertlen}{ O{cm} m }
{
	\dim_to_decimal_in_unit:nn { #2 } { 1 #1 } cm
}
\ExplSyntaxOff
\fi

\usepackage{authblk}

\usepackage[top=2cm, bottom=2cm, left=2cm, right=2cm]{geometry}

\theoremstyle{plain} 
\newtheorem{theorem}{Theorem}[section]
\newtheorem{propositon}{Proposition}[section]
\newtheorem{lemma}[theorem]{Lemma}
\newtheorem{corollary}[theorem]{Corollary}
\newtheorem{remark}[theorem]{Remark}
\newtheorem{definition}[theorem]{Definition}

\theoremstyle{definition} %
\newtheorem{assumption}{Assumption}

\theoremstyle{remark} %

\newcommand{\ignore}[1]{}

\begin{document}
	
	\title{Two-side a posteriori error estimates for the DWR method}
	\author[1]{B. Endtmayer}
	\author[1]{U. Langer}
	\author[2]{T. Wick}

	\affil[1]{Johann Radon Institute for Computational and Applied Mathematics, Austrian Academy of Sciences, Altenbergerstr. 69, A-4040 Linz, Austria}
	\affil[2]{Leibniz Universit\"at Hannover, Institut f\"ur Angewandte
          Mathematik, AG Wissenschaftliches Rechnen, Welfengarten 1, 30167 Hannover, Germany}
	
	\date{}
	
	\maketitle

\begin{abstract}
In this work, we derive two-sided a posteriori error estimates 
for the dual-weighted residual (DWR) method. We 
consider both single and multiple goal functionals. 
Using a saturation assumption, we derive lower bounds yielding 
the efficiency of the error estimator. These results 
hold true for both nonlinear partial differential equations
and nonlinear functionals of interest. Furthermore, the DWR method employed in
this work accounts for balancing the discretization error
with the nonlinear iteration error.
We also perform careful studies of the 
remainder term that is usually neglected. Based on these 
theoretical investigations, several algorithms are designed. 
Our theoretical findings and algorithmic developments are substantiated with some numerical tests.
\end{abstract}

\section{Introduction}
\label{Section: Introduction}
In many applications, 
nonlinear partial differential equations 
must 
be solved. 
Examples can be found in fluid mechanics, fluid structure interaction, solid
mechanics, porous media, fracture/damage mechanics, and electromechanics.  
Specifically in recent years, multiphysics problems in which several phenomena
interact have become quite important due to the advancements of 
computational resources (in particular parallel computing and local mesh
adaptivity). 
However, we are often not interested in the entire solution, 
but in certain functionals of interest,
also called goal functionals. 
Due to the nature of 
multiphysics problems, several goal functionals may be of interest
simultaneously. Motivated by this fact, basic
frameworks for the adaptive treatment of multiple goal functionals were
first proposed in \cite{HaHou03,Ha08}. Recently, other efforts 
have been undertaken in 
\cite{PARDO20101953,AlvParBar2013,KerPruChaLaf2017,BruZhuZwie16,EnWi17, EnLaWi18,EnLaWiPAMM18}.

In these studies, adaptivity is based on a posteriori error estimation, which
	 is  a widely used  and well developed tool in
finite element (FEM)  
computations 
as, for example, presented in  
\cite{BabRhein1978,BankSmith1993,ZienkiewKellyGagoBabu1982,ApSaeWhi96,NeittaanmaekiRepin:2004a,AinsworthOden1997a,Verfuerth:1996a,
  RepinBook2008,Han:2005a,ErikssonEstepHansboJohnson1995}, 
and 
in other discretization techniques too; see, e.g., 
\cite{Angermann1995,PaKwoYou2003,Kleiss:2015:GSP:2795679.2796192,LangerMatculevichRepin:2017b,StolfoRademacherSchroeder2017,WeiWi18}.

The previously mentioned applications are too complicated for a rigorous numerical analysis
that we have in mind. For this reason,
we concentrate on the development of a posteriori error
estimation for a prototype nonlinear stationary setting in this work.
Here, we focus on both fundamental theoretical and practical aspects. Our method 
of choice is goal-oriented error estimation using the dual-weighted residual
	(DWR) method
\cite{BeRa01,BeEsTru2011,RaSu97,OdPru99,BeckerRannacher1995}, which has proven
to be a successful technique. 
In particular, we are interested in the quality of the error estimator.
	Furthermore, it would be desirable to obtain convergence rates for 
	the corresponding adaptive procedure.
Such convergence results  are discussed in
\cite{MoSchwSzeTem06,FeiPraeZee16,HolPo2016,HoPoZhu2015}.
Improvements of convergence rates are discussed in 
\cite{GilPie2000,GilesSuli2002,GilPie2004,ShaOll2018}. 
Concerning upper bounds of the error, we mention the works
\cite{PruOdWeBaBo03,NoSchmiSiebVeeser2006,FeiPraeZee16,AinRan2012}, 
where \cite{PruOdWeBaBo03} also provides a lower bound for 
	the energy norm in the case of linear symmetric elliptic boundary value problems,
and \cite{NoSchmiSiebVeeser2006} for a pointwise error estimate in case of monotone semi-linear problems.

The first goal of this work is to 
	prove
upper and lower bounds 
for both nonlinear 
partial differential equations (PDEs) and nonlinear quantities of interest. 
This is done for a hierarchical approximation in the DWR error
estimator. Hierarchical approaches for the DWR method are also used in
\cite{BaRa03,HaHou03,Ha08,RiWi15_dwr,BruSchweBau18} 
	exploiting higher-order elements.
In this work, we use a partition of unity (PU) localization, which was
developed in \cite{RiWi15_dwr}. 
Here
backwards-integration by parts is not required. 
We can employ the variational form of the error estimator.
Recently, this localization was also applied to  other discretization
techniques like the finite cell method \cite{StolfoRademacherSchroeder2017}
or (boundary element method) BEM-based FEM on polygonal meshes \cite{WeiWi18}.

To 
prove
 the upper and lower bounds for our error estimator, we need a
saturation assumption for the quantity of interest. 
	For other hierarchical based  a posteriori error in  the energy norm,
	this is a widely used assumption \cite {BankWeiser1985,BoErKor1996,BankSmith1993,Verfuerth:1996a},
	where \cite{BoErKor1996} proved that the saturation assumption can be violated pre-asymptotically for certain data.
For some elliptic boundary value problems, 
the saturation assumption is proven 
in the energy norm for small data oscillations;
see e.g., \cite{DoerflerNochetto2002,FerrazOrtnerPraetorius2010,CaGaGed16} 
and \cite{BankParsaniaSauter2013} for hp-FEM, 
  and in \cite{Agouzal2002,AchAchAgou2004} for a modified version of this assumption. 
Furthermore, a proof of the saturation assumption for a convection-diffusion
problem in one dimension is derived in \cite{Rossi2002,KimKim2014}.  
However, we are not aware of results for general goal functionals.
	We notice that
this is even infeasible because the functional 
error can be zero for general goal functionals. 
Therefore, a positive lower bound cannot be obtained.
	A step into this direction was achieved in \cite{RiWi15_dwr}, where 
a common bound for the functional 
error and the error indicators 
could be established for several 
often employed localization techniques.

We emphasize that our previous developments apply to the generalized 
version of the DWR method in which not only the discretization error is
addressed, but also the iteration error can be balanced with the
discretization error \cite{BeJohRa95,ErnVohral2013,MeiRaVih109,RanVi2013,RaWeWo10}.
In particular, following \cite{RanVi2013}, in our previous work  \cite{EnLaWi18}, 
we developed and extended such a framework 
that applies to single and multiple goal functionals.
We notice that, in contrast to these works, we represent the iteration error in
the current paper in a different way, which avoids the solution of the adjoint
problem for checking the adaptive stopping criterion of Newton's method. 
This stopping criteria of Newton's method is also 
 affine-invariant, 
and falls
consequently into the category of Newton schemes discussed in \cite{Deuflhard2011}.
		
The second goal of this work 
consists in the investigation of
the several parts of the DWR error estimator. 
More precisely, we consider: both single and multiple goal functionals,
both the primal and adjoint parts, the iteration error estimator, and
the nonlinear remainder part.
In particular, the latter term is often neglected in the literature.

The outline of this paper is as follows:
In Section \ref{Section: The dual weighted residual method for nonlinear problems }, we introduce the abstract setting and shortly recap the basic concept of the dual weighted residual method. 
Section \ref{Section: Efficiency and reliability results for the DWR estimator in enriched spaces} contains our main result.  
We prove
a lower and upper bound for an error estimator with the additional computable parts as well for the common error estimator under the saturation assumption and a slightly strengthened version, respectively. 
The different 
	parts of the error estimator
and their localization are discussed in Section \ref{Section: Localization and discussions on the error estimator parts}, followed by a discussion for multiple goal functionals in Section \ref{Section: Multigoalfunctionals}. 
In this discussion,  we derive sufficient conditions to avoid error cancellation under our saturation assumption. 
		The resulting algorithms  are in detail  presented 
		in Section
                \ref{Section: Algorithms} for the finite element method. In
                principle, they can 
                also be easily applied to other discretization
                techniques like isogeometric analysis, finite volume methods, 
                finite cell methods,  or virtual element methods.
Section \ref{Section: Numerical examples} provides the results of our numerical experiments.
We performed extensive numerical tests for both single goal and multiple goal functional evaluations
at finite element solutions of the regularized $p$-Laplace equation; 
	see also \cite{DiRu07,Hi2013,ToWi2017}.
Finally, our observations are summarized in Section \ref{Section: Conclusions}.

\section{The dual weighted residual method for nonlinear problems } 	
\label{Section: The dual weighted residual method for nonlinear problems }
	In this section, we briefly recall the abstract setting of our previous work \cite{EnLaWi18}.

\subsection{An abstract setting}
		Let $U$ and $V$ be Banach spaces, and let $\mathcal{A}: U \mapsto V^*$ be a nonlinear operator, 
		where $V^*$ denotes the dual space of the Banach space $V$.
		We consider the 
		primal problem: Find $u \in U$ such that
		\begin{equation} \label{Equation: Cont Primal Problem}
		\mathcal{A}(u)=0 \qquad \text{ in } V^*.
		\end{equation}
Furthermore, we consider finite dimensional subspaces of $U_h \subset U$ and $V_h \subset V$.
	In this paper, $U_h$ and $V_h$ are finite element spaces (we notice, however,
that our ideas are not restricted to a particular discretization method).
This leads to the following finite dimensional problem: Find $u_h \in U_h$ such that
		\begin{equation} \label{Equation: Discret Primal Problem}
				\mathcal{A}(u_h)=0 \qquad \text{ in } V_h^*.
		\end{equation}
		We assume that both (\ref{Equation: Cont Primal Problem}) and (\ref{Equation: Discret Primal Problem}) are solvable. Further assumptions will be imposed la
		ter.
		However,  we are not primarily  interested in a solution of (\ref{Equation: Cont Primal Problem}) itself, 
		 but in one or even several functional evaluations, so called
                 goal functionals, evaluated  at $u\in U$.

\subsection{The dual weighted residual method}
\label{PU-DWR-NONLINEAR-ONEFUNTIONAL}
We now recall the Dual Weighted Residual (DWR) method for nonlinear problems \cite{BeRa01}.
The extensions for balancing the discretization and iteration errors were
undertaken in \cite{RanVi2013,RaWeWo10,MeiRaVih109}. 
In particular, we base our work on \cite{RanVi2013}, where iteration errors of
the nonlinear solver  
were considered. This paper forms together with our previous works
\cite{RiWi15_dwr,EnWi17,EnLaWi18,EnLaWiPAMM18} the basis of the current
study. 
To apply the DWR method, we have to consider the adjoint problem: Find $z \in V$ such that 
	\begin{equation}\label{Equation : cont adjoint Problem}
	\left(\mathcal{A}'(u)\right)^* (z) = J'(u) \qquad \text{ in } U^*,
	\end{equation}
	where $ \mathcal{A}'(u)$ and $J'(u)$ denote the Fr\'echet-derivatives of the nonlinear operator and functional respectively, evaluated at $u$.
	Later we will also need the finite dimensional version of 
	(\ref{Equation : cont adjoint Problem}) that reads as follows:
	Find $z_h \in V_h$ such that 
	\begin{equation}\label{Equation : discrete adjoint Problem}
	\left(\mathcal{A}'(u_h)\right)^* (z_h) = J'(u_h) \qquad \text{ in } U_h^*.
	\end{equation}
Similarly to  the findings in \cite{RanVi2013,BeRa01,RaWeWo10} for the Galerkin case ($U=V$),  
we provide  an error representation in the following theorem: 
\begin{theorem}\label{Theorem: Error Representation}
Let us assume that $\mathcal{A} \in \mathcal{C}^3(U,V)$ and $J \in \mathcal{C}^3(U,\mathbb{R})$. 
If $u$ solves (\ref{Equation: Cont Primal Problem}) 
and $z$ solves (\ref{Equation : cont adjoint Problem}) for $u \in U$, 
then 
	the error representation

			\begin{align} \label{Error Representation}
				\begin{split}
					J(u)-J(\tilde{u})&= \frac{1}{2}\rho(\tilde{u})(z-\tilde{z})+\frac{1}{2}\rho^*(\tilde{u},\tilde{z})(u-\tilde{u}) 
					-\rho (\tilde{u})(\tilde{z}) + \mathcal{R}^{(3)},\nonumber
				\end{split}
			\end{align}
holds true for arbitrary fixed  $\tilde{u} \in U$ and $ \tilde{z} \in V$, where
$\rho(\tilde{u})(\cdot) := -\mathcal{A}(\tilde{u})(\cdot)$,
$\rho^*(\tilde{u},\tilde{z})(\cdot) := J'(u)-\mathcal{A}'(\tilde{u})(\cdot,\tilde{z})$,
			and the remainder term
			\begin{equation}
			\begin{split}	\label{Error Estimator: Remainderterm}
			\mathcal{R}^{(3)}:=\frac{1}{2}\int_{0}^{1}[J'''(\tilde{u}+se)(e,e,e)
			-\mathcal{A}'''(\tilde{u}+se)(e,e,e,\tilde{z}+se^*)
			-3\mathcal{A}''(\tilde{u}+se)(e,e,e)]s(s-1)\,ds,
			\end{split} 
			\end{equation}
			with $e=u-\tilde{u}$ and $e^* =z-\tilde{z}$.	
		\end{theorem}
		\begin{proof}
			We refer the reader to \cite{EnLaWi18} and \cite{RanVi2013} for the details of the proof.
		\end{proof}
Since Theorem \ref{Theorem: Error Representation} is valid for arbitrary $\tilde{z}$ and $\tilde{u}$,
it also holds for the approximations $u_h$ and $z_h$, even if they are not computed exactly.
Thus, the full error estimator reads as 
		\begin{equation} \label{FullErrorEstimator}
		\eta=\frac{1}{2}\rho(\tilde{u})(z-\tilde{z})+\frac{1}{2}\rho^*(\tilde{u},\tilde{z})(u-\tilde{u}) 
		+\rho (\tilde{u})(\tilde{z}) + \mathcal{R}^{(3)}.
		\end{equation} 
This error estimator is exact, however, not computable.		
To obtain a computable error estimator,  we replace $u$ 
by an approximation on enriched finite dimensional spaces $U_h^{(2)}$ and
$V_h^{(2)}$, 
which, for example, was also 
done in  
\cite{GilesSuli2002,BaRa03,BruSchweBau18, RiWi15_dwr, EnWi17, EnLaWi18, EnLaWiPAMM18}.
In our numerical examples
	presented in  Section \ref{Section: Numerical examples},
we use 
	bi-quadratic (2D) finite elements to define the
 enriched spaces  $U_h^{(2)}$ and $V_h^{(2)}$. 
 As 
 in \cite{EnWi17}, spaces with polynomial orders $r>2$ can be adopted as well.
\begin{remark}
Using enriched spaces is expensive. For this reason, already in the early
studies, e.g., \cite{BeRa01,BaRa03,BraackErn02} (patch-wise) interpolations were suggested to approximate ${z}$
and ${u}$. 
\end{remark}

\newpage
\section{Efficiency and reliability results for the DWR estimator}
\label{Section: Efficiency and reliability results for the DWR estimator in enriched spaces}
In this key section, we show efficiency and reliability of a computable DWR estimator 
	in enriched spaces
under a saturation assumption for the goal functional. 
As mentioned in the introduction, this is a widely adopted assumption in
hierarchical based error
 	 estimates; see, e.g.,
\cite{BankWeiser1985,BoErKor1996,BankSmith1993,Verfuerth:1996a}.  
We are not aware of literature satisfying this assumption for general
nonlinear problems and goal functionals.
Furthermore, there might be restrictions to satisfy this condition.  
For error estimates in the energy norm, an analysis regarding this assumption
can be found in
\cite{DoerflerNochetto2002,Agouzal2002,AchAchAgou2004,FerrazOrtnerPraetorius2010,BankParsaniaSauter2013,CaGaGed16,ErathGanterPraetorius2018} for
linear elliptic boundary value problems 
depending on the oscillation of the data.
Finally, we employ higher-order corrections of the error estimator.
Similar ideas correcting the functional value were discussed in 
\cite{GilesSuli2002,Giles2008inpro,ShaOll2018}.
	Such techniques have also been used
to derive an upper bound of the error without using the saturation assumption 
in \cite{NochettoVeeserVerani2009,AinRan2012, LadPleCha2013}. 
Lower and upper bounds were established 
for symmetric linear elliptic boundary value problems in \cite{PruOdWeBaBo03},
and 
for monotone and semi-linear problems for point-wise error estimates  in \cite{NoSchmiSiebVeeser2006}.

\subsection{Preliminary results}	
We now first recall some notation and known statements.
Let $u_h^{(2)} \in U_h^{(2)}$ be the exact solution of the discretized primal problem
$\mathcal{A}(u_h^{(2)})=0$ in $(V_h^{(2)})^* $,
and $z_h^{(2)} \in V_h^{(2)}$ the exact solution of the discretized adjoint problem	
$(\mathcal{A'}(u_h^{(2)}))^*(z_h^{(2)}) =J'(u_h^{(2)})$ in $(U_h^{(2)})^*.$

\begin{corollary}\label{Corrollary: discrete Error Representation}	
Let the assumptions of Theorem \ref{Theorem: Error Representation} be fulfilled. 		
Then the error representation
		\begin{align} \label{discrete Error Representation}
			\begin{split}
				J(u_h^{(2)})-J(\tilde{u})&= \frac{1}{2}\rho(\tilde{u})(z_h^{(2)}-\tilde{z})+\frac{1}{2}\rho^*(\tilde{u},\tilde{z})(u_h^{(2)}-\tilde{u}) 
				-\rho (\tilde{u})(\tilde{z}) + \mathcal{R}^{(3)(2)} \nonumber
			\end{split}
		\end{align}
holds for arbitrary but fixed  $\tilde{u} \in U_h^{(2)}$ and $ \tilde{z} \in V_h^{(2)}$,
where
    $\rho(\tilde{u})(\cdot) := -\mathcal{A}(\tilde{u})(\cdot)$,
    $\rho^*(\tilde{u},\tilde{z})(\cdot) := J'(\tilde{u})-\mathcal{A}'(\tilde{u})(\cdot,\tilde{z})$,
    and 
    $\mathcal{R}^{(3)(2)}:=\frac{1}{2}\int_{0}^{1}[J'''(\tilde{u}+se^{(2)})(e^{(2)},e^{(2)},e^{(2)})
				-\mathcal{A}'''(\tilde{u}+se^{(2)})(e^{(2)},e^{(2)},e^{(2)},\tilde{z}+se^{(2),*})
				-3\mathcal{A}''(\tilde{u}+se^{(2)})(e^{(2)},e^{(2)},e^{(2),*})]s(s-1)\,ds$
    denotes the remainder term,
    with $e^{(2)}=u_h^{(2)}-\tilde{u}$ and $e^{(2),*} =z_h^{(2)}-\tilde{z}$.

\begin{proof}
		The statement  follows immediately from Theorem \ref{Theorem: Error Representation}.
	\end{proof}
\end{corollary}
	\begin{remark}
		For a linear problem and a functional fulfilling $J'''=0$ this theorem allows us to compute $J(u_h^{(2)})$ without the computation of $ u_h^{(2)}$ since $\rho(\tilde{u})(z_h^{(2)}-\tilde{z})=\rho^*(\tilde{u},\tilde{z})(u_h^{(2)}-\tilde{u}) $ for linear problems as already stated in \cite{GilesSuli2002}.
	\end{remark}

Replace $u$ and $z$ by the approximations $u_h^{(2)}$ and $z_h^{(2)}$ in (\ref{FullErrorEstimator}), we get the computable error estimator
\begin{equation} \label{Definition: Error Estimator}
		\eta^{(2)}:=\frac{1}{2}\rho(\tilde{u})(z_h^{(2)}-\tilde{z})+\frac{1}{2}\rho^*(\tilde{u},\tilde{z})(u_h^{(2)}-\tilde{u}) 
		+\rho (\tilde{u})(\tilde{z}) + \mathcal{R}^{(3)(2)}.
	\end{equation}	
Now, Corollary~\ref{Corrollary: discrete Error Representation} together with
(\ref{Definition: Error Estimator}) allows us to  recover the error 
$J(u_h^{(2)})-J(\tilde{u})$.
A similar representation of the error   $J(u_h^{(2)})-J(\tilde{u})$ is derived in \cite{GilesSuli2002,Giles2008inpro,GilPie2004}.

\subsection{Efficiency and reliability of the DWR estimator using a saturation assumption}
The following lemma provides a two-side estimate of the modulus of 
$\eta^{(2)}$ defined by (\ref{Definition: Error Estimator}).

\begin{lemma}\label{Lemma:  Error Estimatorboundsremainder}
    Under the assumptions of Theorem \ref{Theorem: Error Representation}, 
    the two-side estimate
		\begin{equation}
			| J(u)-J(\tilde{u}) |-|J(u)-J(u_h^{(2)})| \leq |\eta^{(2)}| \leq | J(u)-J(\tilde{u}) |+ |J(u)-J(u_h^{(2)})|. \nonumber
		\end{equation}
    holds for the computable error estimator $\eta^{(2)}$.

\begin{proof}
From
			$	|\eta| = |\eta^{(2)} - (\eta ^{(2)}- \eta)| $,
			we can deduce that 
			\begin{align*}
				|\eta| - |\eta - \eta^{(2)}| \leq |\eta^{(2)}| \leq 	|\eta|+|\eta - \eta^{(2)}|.
			\end{align*}
Since $U_h^{(2)}$ is an enriched space, we have $U_h \subset U_h^{(2)} \subset U$. 
It follows that  $\eta-\eta^{(2)}=J(u)-J(\tilde{u})-J(u_h^{(2)})+J(\tilde{u}) = J(u)-J(u_h^{(2)})$, 
which leads us together with $\eta= J(u)-J(\tilde{u})$ to the estimates stated in the lemma.
\end{proof}
\end{lemma}

\begin{assumption}[Saturation assumption for the goal functional] \label{Assumption: Better approximation}
Let $u_h^{(2)}$ solve the primal problem on $U_h^{(2)}$ and let $\tilde{u}$ be some approximation. 
Then we assume that 
\[
|J(u)-J(u_h^{(2)})| < b_h| J(u)-J(\tilde{u}) |
\]
for some  $b_h<b_0$ and some fixed $b_0 \in (0,1)$.
\end{assumption}
	
\begin{theorem} \label{Theorem: Efficiency and Reliability with remainder}
Let the saturation Assumption~\ref{Assumption: Better approximation} be fulfilled. 
Then the  computable error estimator $\eta^{(2)}$ satisfies the efficiency and reliability estimates
\begin{equation}
\label {Estimate: hEffektivity+Remainder}
 \underline{c}_h|\eta^{(2)}| \leq | J(u)-J(\tilde{u}) | \leq 	\overline{c}_h|\eta^{(2)}|
 \quad \mbox{and} \quad
 \underline{c}|\eta^{(2)}| \leq | J(u)-J(\tilde{u}) | \leq 	\overline{c}|\eta^{(2)}|, 
\end{equation}
with the positive constants $\underline{c}_h:= 1/(1+b_h)$, $\overline{c}_h:=1/( 1-b_h)$, $\underline{c}:= 1/(1+b_0)$,
and $\overline{c}:=1/( 1-b_0)$. 
%
\begin{proof}
			In the proof of Lemma \ref{Lemma:  Error Estimatorboundsremainder}, we concluded that 	$|\eta| - |\eta - \eta^{(2)}| \leq |\eta^{(2)}| \leq 	|\eta|+|\eta - \eta^{(2)}|$ which is equivalent to the statement that 
			$	|\eta^{(2)}| - |\eta^{(2)} - \eta| \leq |\eta| \leq 	|\eta^{(2)}|+|\eta^{(2)} - \eta|$. 
			Therefore, we have 
			\begin{align*}
				|\eta^{(2)}| - |J(u)-J(u_h^{(2)})| \leq &| J(u)-J(\tilde{u}) | \leq 	|\eta^{(2)}|+|J(u)-J(u_h^{(2)})|,\nonumber 
			\end{align*}		
			which together with Assumption~\ref{Assumption: Better approximation} immediately 
			yield the first inequalities in (\ref{Estimate: hEffektivity+Remainder}).
			The second statement follows from $\underline{c} \leq \underline{c}_h$  
			and  $\overline{c}_h \leq \overline{c}$  due to $b_h < b_0$.
\end{proof}
\end{theorem}
	\begin{remark}
			The left estimate in  (\ref{Estimate: hEffektivity+Remainder}) also holds for $b_0 \in (0,1]$, which is  called weak saturation assumption  in the case of energy norm estimates; see \cite{CaGaGed16}.
	\end{remark}
		
Now let us assume that we neglect the remainder term $\mathcal{R}^{(3)(2)}$ and iteration error estimator $\rho(\tilde{u})(\tilde{z})$ in the error estimator $\eta^{(2)}$. 
	This gives the practical error estimator 
	\begin{align} \label{Error Estimator: practical discretization wo Remainder}
		\eta_h^{(2)}:=\frac{1}{2}\rho(\tilde{u})(z_h^{(2)}-\tilde{z})+\frac{1}{2}\rho^*(\tilde{u},\tilde{z})(u_h^{(2)}-\tilde{u}) ,
	\end{align}
	where the corresponding theoretical error estimator is given by 
	\begin{align} \label{Error Estimator: theoretical discretization wo Remainder}
		\eta_h:=\frac{1}{2}\rho(\tilde{u})(z-\tilde{z})+\frac{1}{2}\rho^*(\tilde{u},\tilde{z})(u-\tilde{u}) .
	\end{align}
	Variants of these error estimators are discussed, e.g.,  in \cite{BeRa01,RanVi2013,RiWi15_dwr}; also see the references therein. 
	
\begin{lemma} \label{Lemma: Difference Error Estimators}
		Let $\eta_h$ be defined as in (\ref{Error Estimator: theoretical discretization wo Remainder}),
		and $\eta_h^{(2)}$ be defined as in (\ref{Error Estimator: practical discretization wo Remainder}). 
		Furthermore, let us assume that the assumptions of Theorem \ref{Theorem: Error Representation} are fulfilled . 
		Then, for the exact solutions $u_h^{(2)}$ and $z_h^{(2)}$ from the spaces $U_h^{(2)}$ and $V_h^{(2)}$,
		the following two-side estimates
		\begin{equation}\label{Difference Error Estimators1}
			|J(u)-J(u_h^{(2)})|- |\mathcal{R}^{(3)} -\mathcal{R}^{(3)(2)} |\leq |\eta_h -\eta_h^{(2)}| \leq |J(u)-J(u_h^{(2)})|+ |\mathcal{R}^{(3)} -\mathcal{R}^{(3)(2)} |,
		\end{equation} 
		and 
		\begin{equation}\label{Difference Error Estimators2}
			|J(u)-J(\tilde{u})|- |\rho (\tilde{u})(\tilde{z})|-|\mathcal{R}^{(3)}| \leq |\eta_h| \leq |J(u)-J(\tilde{u})|+|\rho (\tilde{u})(\tilde{z})|+|\mathcal{R}^{(3)}|,
		\end{equation} 
		hold,
		with $\mathcal{R}^{(3)}$ defined in (\ref{Error Estimator: Remainderterm}) 
		and $\mathcal{R}^{(3)(2)}$ from Corollary~\ref{Corrollary: discrete Error Representation}.
\end{lemma}
\begin{proof}
			From Theorem \ref{Theorem: Error Representation}, we know that 
			\begin{equation*}
				J(u)-J(\tilde{u})= \underbrace{\frac{1}{2}\rho(\tilde{u})(z-\tilde{z})+\frac{1}{2}\rho^*(\tilde{u},\tilde{z})(u-\tilde{u}) }_{\eta_h}
				+\rho (\tilde{u})(\tilde{z}) + \mathcal{R}^{(3)},
			\end{equation*}
			and Corollary~\ref{Corrollary: discrete Error Representation} provides us with the identity
			\begin{equation*}
				J(u_h^{(2)})-J(\tilde{u})=\underbrace{ \frac{1}{2}\rho(\tilde{u})(z_h^{(2)}-\tilde{z})+\frac{1}{2}\rho^*(\tilde{u},\tilde{z})(u_h^{(2)}-\tilde{u}) }_{\eta_h^{(2)}}
				+\rho (\tilde{u})(\tilde{z}) + \mathcal{R}^{(3)(2)}.
			\end{equation*}							
			These two identities imply the identity		
				$J(u)-J(u_h^{(2)}) =  \eta_h - \eta_h^{(2)} + \mathcal{R}^{(3)}-\mathcal{R}^{(3)(2)}$.		
			We now conclude that 
				$|J(u)-J(u_h^{(2)})-\mathcal{R}^{(3)}+\mathcal{R}^{(3)(2)}| = | \eta_h - \eta_h^{(2)}|$,
			from which we immediately get the inequalities (\ref{Difference Error Estimators1}).
			The second statement follows directly from Theorem \ref{Theorem: Error Representation}.
\end{proof}

\begin{lemma}\label{Lemma: Bounds for Error Estimator}
		Under the conditions of Lemma \ref{Lemma: Difference Error Estimators},
		inequalities
		\begin{align}
		\label{Bounds for Error Estimator}
			|\eta_h^{(2)}| -\gamma(\mathcal{A},J,u_h^{(2)},u,\tilde{u}) \leq |J(u)-J(\tilde{u})| \leq |\eta_h^{(2)}| +\gamma(\mathcal{A},J,u_h^{(2)},u,\tilde{u})
		\end{align}
		are valid, where 
\begin{equation}
\label{twick_Nov_13_2018_eq_1}
\gamma(\mathcal{A},J,u_h^{(2)},u,\tilde{u}): =|J(u)-J(u_h^{(2)})|+ |\mathcal{R}^{(3)} -\mathcal{R}^{(3)(2)} |+|\rho (\tilde{u})(\tilde{z})|+|\mathcal{R}^{(3)}|.
\end{equation}
\end{lemma}		
\begin{proof}
			Inequalities (\ref{Bounds for Error Estimator})  immediately follow from (\ref{Difference Error Estimators1}), (\ref{Difference Error Estimators2}) and 
			\begin{equation}
				|\eta_h|-|\eta_h -\eta_h^{(2)}| \leq |\eta_h^{(2)}| \leq |\eta_h|+|\eta_h -\eta_h^{(2)}|.\nonumber 
			\end{equation}
\end{proof}

 \subsection{Practicable error estimator under a strengthened saturation assumption}

We refine our previous 
analysis
in order to derive a similar statement for the
practicable error estimator $\eta_h^{(2)}$. 
We suppose the following strengthened saturation assumption:
%
\begin{assumption}[Strengthened saturation assumption for the goal functional]  \label{Assumption: With Remainder}
		Let $u_h^{(2)}$ solve the primal problem on $U_h^{(2)}$, 
		and let $\tilde{u}$ be some approximation. 
Then we assume that the inequality
\[
\gamma(\mathcal{A},J,u_h^{(2)},u,\tilde{u}) < b_{h, \gamma}|
J(u)-J(\tilde{u}) |
\]
with $\gamma(\cdot)$ defined in (\ref{twick_Nov_13_2018_eq_1}),
holds true
for some $b_{h, \gamma}<b_{0, \gamma}$ with some fixed $b_{0, \gamma} \in (0,1)$.
\end{assumption}
\begin{remark}
		Of course, Assumption~\ref{Assumption: With Remainder} implies Assumption~\ref{Assumption: Better approximation}. 
		If, on the other hand,  Assumption~\ref{Assumption: Better approximation} holds, then  Assumption~\ref{Assumption: With Remainder} is fulfilled up to higher-order terms  ($ |\mathcal{R}^{(3)} -\mathcal{R}^{(3)(2)} |$, $|\mathcal{R}^{(3)}|$), 
		and the part $|\rho (\tilde{u})(\tilde{z})|$, which can be controlled by the accuracy of the nonlinear solver.
	\end{remark}

\begin{theorem} \label{Theorem: Efficiency and Reliability without remainder}
Let the saturation Assumption~\ref{Assumption: With Remainder} be fulfilled. 
Then the  practical error estimator $\eta_h^{(2)}$ satisfies the efficiency and reliability estimates
\begin{equation}
\label {Estimate: hEffektivity-Remainder}
 \underline{c}_{h, \gamma}|\eta_h^{(2)}| \leq | J(u)-J(\tilde{u}) | \leq 	\overline{c}_{h, \gamma}|\eta_h^{(2)}|
 \quad \mbox{and} \quad
 \underline{c}_{\gamma}|\eta_h^{(2)}| \leq | J(u)-J(\tilde{u}) | \leq 	\overline{c}_{\gamma}|\eta_h^{(2)}| , 
\end{equation}
with the 
positive 
constants $\underline{c}_{h, \gamma}:= 1/(1+b_{h, \gamma})$, $\overline{c}_{h, \gamma}:=1/( 1-b_{h, \gamma})$, $\underline{c}_{\gamma}:= 1/(1+b_{0, \gamma})$, 
$\overline{c}_{\gamma}:=1/( 1-b_{0, \gamma})$. 
%
%
\begin{proof}
From Lemma \ref{Lemma: Bounds for Error Estimator}, we concluded that 	$|\eta_h^{(2)}| -\gamma(\mathcal{A},J,u_h^{(2)},u,\tilde{u}) \leq |J(u)-J(\tilde{u})| \leq |\eta_h^{(2)}| +\gamma(\mathcal{A},J,u_h^{(2)},u,\tilde{u})$ which  together with Assumption~\ref{Assumption: With Remainder} imply 
that $$\frac{1}{1+b_{h, \gamma}}|\eta_h^{(2)}| \leq |J(u)-J(\tilde{u})| \leq \frac{1}{1-b_{h, \gamma}}|\eta_h^{(2)}|.$$ This is our first statement.			
				Like in the proof of Theorem \ref{Theorem: Efficiency and Reliability with remainder}, 
				the second statement follows from  $\underline{c}_{\gamma} \leq \underline{c}_{h,
					\gamma}$ and  $\overline{c}_{h,\gamma} \leq \overline{c}_\gamma$.
					We mention that $b_{h,\gamma}< b_{0,\gamma}$.
\end{proof}
	\end{theorem}
\begin{remark}
	The left estimate  in (\ref{Estimate: hEffektivity-Remainder}) is also true for $b_{0,\gamma} \in (0,1]$.
\end{remark}

\subsection{Bounds of the effectivity indices}
We finally derive bounds for the effectivity indices $I_{eff}$ and $I_{eff,\gamma}$ defined by the relations
\begin{equation*}
 I_{eff}:= \frac{|\eta^{(2)}|}{|J(u)-J(\tilde{u})|} 
\quad \mbox{and} \quad
 I_{eff,\gamma}:= \frac{|\eta_h^{(2)}|}{|J(u)-J(\tilde{u})|},
\end{equation*}
respectively.

\begin{theorem}[Bounds on the Effectivity Index] \label{Theorem: Ieffbounds}
			Let the assumptions of Theorem \ref{Theorem: Error Representation} be fulfilled.
			Then the following two statements are true:
			\begin{enumerate}
				\item If Assumption~\ref{Assumption: Better approximation} is fulfilled,  then $I_{eff} \in [1-b_0,1+b_0]$, and	
				if additionally $b_h \rightarrow 0$, then  $I_{eff} \rightarrow 1$.
				\item If Assumption~\ref{Assumption: With Remainder} is fulfilled,  
				then $I_{eff,\gamma} \in [1-b_{0,\gamma},1+b_{0,\gamma}]$, and  if additionally 
				$b_{h,\gamma} \rightarrow 0$, then  $I_{eff,\gamma} \rightarrow 1$.
			\end{enumerate}
\end{theorem}
\begin{proof}
The first statement follows from Lemma \ref{Lemma:  Error Estimatorboundsremainder} and Assumption~\ref{Assumption: Better approximation}, whereas  the second statement is obtained from Lemma \ref{Lemma: Bounds for Error Estimator} and Assumption~\ref{Assumption: With Remainder} in the same way.
\end{proof}
%
\begin{remark}
			We notice that $I_{eff,\gamma}  \rightarrow 1$ was also already observed in \cite{BaRa03}  and proven for smooth adjoint solutions in the linear case.
\end{remark}
		
\begin{propositon}
			If $J''' \equiv 0$ and if $A''$ is of the form $A''(u)\equiv B u +C$ for some linear operator $B$ and some $C$ not depending on $u$, then we have the representation
			\begin{equation}
				\mathcal{R}^{(3)}= \frac{1}{24}\left(3(B(u + \tilde{u}))(e,e,e^*) + (Be)(e,e,z+\tilde{z})\right) + \frac{1}{4} C(e,e,e^*). \nonumber
			\end{equation}
\end{propositon}
\begin{remark}
			In this section we did not consider the error contributions from the approximation of the data (source terms, boundary conditions) and  quadrature formulas.
\end{remark}


\section{Localization and discussions of the error estimator parts} 
\label{Section: Localization and discussions on the error estimator parts}
		
In this section,  we further discuss the computable error estimator $\eta^{(2)}$ defined in (\ref{Definition: Error Estimator}).
We separate the error estimator $\eta^{(2)}$ into the following three parts $\eta_h^{(2)}$, $\eta_k$,  and
$\eta^{(2)}_\mathcal{R}$ as follows:
%
		\begin{equation}
		\eta^{(2)}:=\underbrace{\frac{1}{2}\rho(\tilde{u})(z_h^{(2)}-\tilde{z})+\frac{1}{2}\rho^*(\tilde{u},\tilde{z})(u_h^{(2)}-\tilde{u}) }_{:= \eta_h^{(2)}}
		+\underbrace{\rho (\tilde{u})(\tilde{z})}_{:=\eta_k} + \underbrace{\mathcal{R}^{(3)(2)}}_{:=\eta^{(2)}_\mathcal{R}}. \nonumber
		\end{equation}

\paragraph{The first  part $\eta_h^{(2)}$ of the error estimator $\eta^{(2)}$:} 
		Following
		\cite{RanVi2013}, we relate the discretization error to $\eta_h^{(2)}$. 
		We use the partition of unity approach developed in \cite{RiWi15_dwr} to localize $\eta_h^{(2)}$. 
		This means that we choose a set of functions $\{\psi_1, \psi_2,  \cdots,\psi_N\}$ 
		(a typical choice would be the finite element basis functions)  such that $	\sum_{i=1}^{N} \psi_i \equiv 1$. 
		Therefore, we have the representation
			\begin{equation*} \label{Local Error Estimator: discretization}
			\eta_{h}^{(2)}:=\sum_{i=1}^{N}\eta_i,
			\end{equation*}
			with 
			\begin{equation}
			\label{eta_i_PU}
			\eta_i:=\frac{1}{2}\rho(\tilde{u})((z_h^{(2)}-\tilde{z})\psi_i)+\frac{1}{2}\rho^*(\tilde{u},\tilde{z})((u_h^{(2)}-\tilde{u})\psi_i).
			\end{equation}
		However, in contrast to our previous work \cite{EnLaWi18}, we emphasize that we do not replace $\tilde{z}$ by $i_hz_h^{(2)}$.
		In our numerical examples, we choose conforming bilinear elements $Q_1^c$ for our partition of unity.
		Furthermore, we distribute the error contributions contained in hanging nodes 
			in a way that is different from
		our previous work. For the partition of unity used in our numerical experiments, we distribute the error as in our previous work, however, splitting the error in the hanging nodes into two equal parts and add the distribution to the neighboring nodes which belong to coarser element, as illustrated in Figure~\ref{figure: distributionErrorPUHangingNodes}.
			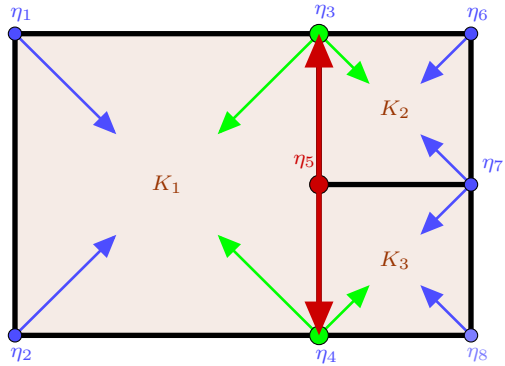
\begin{figure}[H]
				\centering
				\definecolor{ccqqqq}{rgb}{0.8,0,0}
				\definecolor{xdxdff}{rgb}{0.49019607843137253,0.49019607843137253,1}
				\definecolor{zzttqq}{rgb}{0.6,0.2,0}
				\definecolor{qqffqq}{rgb}{0,1,0}
				\definecolor{ududff}{rgb}{0.30196078431372547,0.30196078431372547,1}
				\scalebox{1.0}{
				\begin{tikzpicture}[line cap=round,line join=round,>=triangle 45,x=1cm,y=1cm]
				\clip(-5,-1) rectangle (3,5);
				\fill[line width=2pt,color=zzttqq,fill=zzttqq,fill opacity=0.10000000149011612] (-4,4) -- (-4,0) -- (0,0) -- (0,4) -- cycle;
				\fill[line width=2pt,color=zzttqq,fill=zzttqq,fill opacity=0.10000000149011612] (0,0) -- (2,0) -- (2,2) -- (0,2) -- cycle;
				\fill[line width=2pt,color=zzttqq,fill=zzttqq,fill opacity=0.10000000149011612] (2,2) -- (2,4) -- (0,4) -- (0,2) -- cycle;
				\draw [line width=2pt,color=black] (-4,4)-- (-4,0);
				\draw [line width=2pt,color=black] (-4,0)-- (0,0);
				\draw [line width=2pt,color=black] (0,0)-- (0,4);
				\draw [line width=2pt,color=black] (0,4)-- (-4,4);
				\draw [line width=2pt,color=black] (0,0)-- (2,0);
				\draw [line width=2pt,color=black] (2,0)-- (2,2);
				\draw [line width=2pt,color=black] (2,2)-- (0,2);
				\draw [line width=2pt,color=black] (2,2)-- (2,4);
				\draw [line width=2pt,color=black] (2,4)-- (0,4);
				\draw [line width=2pt,color=black] (0,4)-- (0,2);
			
				\draw [color=qqffqq][->,line width=1pt] (0,4) -- (0.6667,3.3333);
				\draw [color=ududff] [->,line width=1pt] (2,4) -- (1.3333,3.3333);
				\draw [color=ududff] [->,line width=1pt] (2,2) -- (1.3333,2.6667);
				\draw [color=qqffqq][->,line width=1pt] (0.0,-0.0) -- (0.6667,0.6667);
				\draw [color=ududff] [->,line width=1pt] (2,0) -- (1.3333,0.6667);
				\draw [color=ududff]  [->,line width=1pt] (2,2) -- (1.3333,1.3333);
				\draw [color=ududff]  [->,line width=1pt] (-4,4) -- (-2.6667,2.6667);
				\draw [color=qqffqq][->,line width=1pt] (0,4) -- (-1.3333,2.6667);
				\draw [color=ududff] [->,line width=1pt] (-4,0) -- (-2.6667,1.3333);
				\draw [color=qqffqq][->,line width=1pt] (0.0,-0.0) -- (-1.3333,1.3333);
				\begin{scriptsize}
				\draw [fill=ududff] (-4,4) circle (2.5pt);
				\draw[color=ududff] (-3.9092991945462403,4.251141456382161) node {$\eta_1$};
				\draw [fill=ududff] (-4,0) circle (2.5pt);
				\draw[color=ududff] (-3.9092991945462403,-0.2530188159990778) node {$\eta_2$};
				\draw [fill=qqffqq] (0,4) circle (3.5pt);
				\draw[color=ududff] (0.08882344583683537,4.298316944764263) node {$\eta_3$};
				\draw [fill=xdxdff] (2,0) circle (2.5pt);
				\draw[color=xdxdff] (2.093781702076136,-0.2530188159990778) node {$\eta_8$};
				\draw [fill=ududff] (2,2) circle (2.5pt);
				\draw[color=ududff] (2.293781702076136,2.2579770722383823) node {$\eta_7$};
				\draw [fill=ccqqqq] (0,2) circle (3.5pt);
				\draw[color=ccqqqq] (-0.18882344583683537,2.3051525606204835) node {$\eta_5$};
				\draw [fill=ududff] (2,4) circle (2.5pt);
				\draw[color=ududff] (2.093781702076136,4.251141456382161) node {$\eta_6$};
				\draw [fill=qqffqq] (0,0) circle (3.5pt);
				\draw[color=ududff] (0.10061731793236067,-0.28122494390355248) node {$\eta_4$};
				\draw[color=zzttqq] (-2,2) node {$K_1$};
				\draw[color=zzttqq] (1,3) node {$K_2$};
				\draw[color=zzttqq] (1,1) node {$K_3$};
					\draw [line width=2pt,color=ccqqqq] [->,line width=2pt] (0,2) -- (0,4);
					\draw [line width=2pt,color=ccqqqq] [->,line width=2pt] (0,2) -- (0,0);			
				\end{scriptsize}
				\end{tikzpicture}							
			}
			\caption{Distribution of the error contribution in a hanging node (red) to the neighboring nodes on the coarser element (green) for $Q^1_c$ basis functions as partition of unity.}\label{figure: distributionErrorPUHangingNodes}
\end{figure}

\paragraph{The second  part $\eta_k$ of the error estimator $\eta^{(2)}$:} 
    The second part, $\eta_k=\rho(\tilde{u})(\tilde{z})$, is related to the iteration error as in \cite{RanVi2013}. 
    Therefore, we can use this quantity as stopping rule for the nonlinear solver, e.g., for Newton's Method.
		In \cite{EnLaWi18} and \cite{RanVi2013},  $\tilde{z}$ was computed in every Newton step in order to evaluate the stopping criteria. If we further follow the path in \cite{EnLaWi18}, and do not compare the iteration error to the current discretization error as in \cite{RanVi2013}, but to the discretization error of the previous mesh, we can use the following Lemma to reduce the computational cost.

\begin{lemma}
    Let 
    $\tilde{u}$ be an arbitrary 
    element 
    from $U$,
    and
    $  \delta\tilde{u} \in U$ be the solution of the problem: Find  $  \delta\tilde{u} \in U$ such that
				\begin{equation} \label{Equation: Linearized Problem stopping rule adjoint}
					\mathcal{A'}(\tilde{u})(\delta\tilde{u} ,v) =	-\mathcal{A}(\tilde{u})(v)\qquad  \forall v \in V,
				\end{equation}
				and $\hat{z} \in V$ be the solution of the  problem: Find $\hat{z} \in V$  such that 
				\begin{equation} \label{Equation: Linearized Problem stopping rule primal}
					\mathcal{A'}(\tilde{u})(v,\hat{z}) =J'(\tilde{u})(v)\qquad  \forall v \in U.
				\end{equation}
				Then we have the equation 
				$-\mathcal{A}(\tilde{u})(\hat{z})=J'(\tilde{u})(\delta \tilde{u})$. 
\end{lemma}
\begin{proof} It is trivial to see that
						$-\mathcal{A}(\tilde{u})(\hat{z})=\mathcal{A'}(\tilde{u})(\delta\tilde{u} ,\hat{z}) =J'(\tilde{u})(\delta \tilde{u})$.
\end{proof}

			\begin{remark}
				This means that, instead of solving the adjoint problem, we can solve for the upcoming Newton update in advance.
				This only holds true if the Newton update $\delta \tilde{u}$ and  $\hat{z}$ are the exact solutions of (\ref{Equation: Linearized Problem stopping rule primal}) and (\ref{Equation: Linearized Problem stopping rule adjoint}), respectively.
			\end{remark}

\paragraph{The third  part $\eta^{(2)}_\mathcal{R}$ of the error estimator $\eta^{(2)}$:} 
		The third part $R^{(3)(2)}$ was neglected in \cite{EnLaWi18}. 
		We localize this error by the local contributions of this error estimator parts computed on the elements. 
		This leads to the local remainder 
			\begin{equation} \label{Equation: LocalErrorContributionRemainder}
			\eta_{\mathcal{R},K}^{(2)} := \mathcal{R}^{(3)(2)}_{|K}, 
			\end{equation}
		in third error estimator part 	on the element $K$.
			Alternatively, one could also use again the partition of unity approach, which was discussed for the first part.


\section{Multiple goal functionals}
\label{Section: Multigoalfunctionals}
		For completeness of presentation we shortly recall the
                multigoal approach presented in \cite{EnLaWi18}.
                From a general  point of view,  it may be questionable whether this
                approach is computationally interesting in comparison to
                	the use of
                uniform mesh refinement. 
                However,
                our previous studies have 
                shown excellent results. Moreover, this approach has still the
                advantage that we have an error estimator (and not only
                indicators for mesh refinement) providing
                us concrete quantitative numbers that are useful as stopping criteria or
                error information engineering applications.

In the following, we  assume that we are interested in the evaluation of $N$ functionals, which we denote by
		$J_1, J_2, \ldots,$ $J_{N-1}$, and $J_{N}$.
		We already derived how to compute local error estimators for a single functional. It would be possible to compute the local error contribution of all $N$ functionals separately, and add them up afterwards. 
		However,  
		we would have to solve $N$ adjoint problems in this case.
		Therefore, we follow the idea in \cite{HaHou03,Ha08} to combine
                the goal functionals. 
		To this end, we assume that a solution $u$ of problem (\ref{Equation: Cont Primal Problem}) and the chosen $\tilde{u} \in U$ belong to $\bigcap_{i=1}^N \mathcal{D}(J_i)$, where $\mathcal{D}(J_i)$ describes the domain of $J_i$. 
		\begin{definition}[error-weighting function \cite{EnLaWi18}]
			Let $ M \subseteq \mathbb{R}^N$. 
			We say that $\mathfrak{E}: (\mathbb{R}^+_0)^N \times  M \mapsto \mathbb{R}^+_0$ is an \textit{error-weighting
				function} if  $\mathfrak{E}(\cdot,m) \in
			\mathcal{C}^1((\mathbb{R}^+_0)^N,\mathbb{R}^+_0)$ is strictly monotonically
			increasing in each component and $\mathfrak{E} (0,m)=0$ for all
			$m \in M$.
		\end{definition}
		Let us define  $\vec{J}: \bigcap_{i=1}^N \mathcal{D}(J_i) \subseteq U \mapsto \mathbb{R}^N$  as $\vec{J}(v):=(J_1(v),J_2(v),\cdots, J_{N}(v) )$ for all $v \in \bigcap_{i=1}^N \mathcal{D}(J_i)$. Furthermore, we define the operation $|\cdot|_N:\mathbb{R}^N\mapsto (\mathbb{R}^+_0)^N$ as $|x|_N:= (|x_1|,|x_2|,\cdots,|x_N|)$ for $x \in \mathbb{R}^N $.
		Following \cite{EnLaWi18}, the \text{error functional} is given by
		\begin{align}
			\tilde{J}_{\mathfrak{E}}(v):=\mathfrak{E}(|\vec{J}(u)-\vec{J}(v)|_N, \vec{J}(\tilde{u})) \qquad \forall v \in \bigcap_{i=1}^N \mathcal{D}(J_i). \nonumber
		\end{align}

	Of course, the exact solution $u$ is not known. 
	Therefore,
	$\tilde{J}_{\mathfrak{E}}$ cannot be computed. As for the error estimate itself,  we use the approximation $u_h^{(2)}$ in the enriched space instead of an exact solution $u$  to approximate $\tilde{J}_{\mathfrak{E}}$ and $J_{\mathfrak{E}}$.  This finally reads as follows
	\begin{align}	\label{ErrorrepresentationFunctionalapprox}
		J_{\mathfrak{E}}(v):=\mathfrak{E}(|\vec{J}(u_h^{(2)})-\vec{J}(v)|_N, \vec{J}(\tilde{u})) \qquad \forall v \in \bigcap_{i=1}^N \mathcal{D}(J_i).
	\end{align}
\begin{propositon}
If Assumption~\ref{Assumption: Better approximation} is fulfilled 
for $\tilde{u}_1$ and $\tilde{u}_2$,
and if 
\[
J_i(u_h^{(2)}) ~\not \in ~[J_i(\tilde{u}_1),J_i(\tilde{u}_2)]
~\cup~[J_i(\tilde{u}_2),J_i(\tilde{u}_1)],
\]
for all $J_i$, $i=1,\ldots, N$, 
then we avoid  error cancellation, i.e, 
if
$
|J_i(u)-J_i(\tilde{u}_1)| \leq
|J_i(u)-J_i(\tilde{u}_2)|  \quad\forall i \in \{1, \cdots, N\},
$
then
$
J_{\mathfrak{E}}(\tilde{u}_1) \leq J_{\mathfrak{E}}(\tilde{u}_2). 
$
\end{propositon}

\begin{proof} 
For $\tilde{J}_{\mathfrak{E}}$, it is clear that 
\[
\forall i \in \{1, \cdots, N\} \qquad |J_i(u)-J_i(\tilde{u}_1)| 
\leq |J_i(u)-J_i(\tilde{u}_2)| \implies    
\tilde{J}_{\mathfrak{E}}(\tilde{u}_1) \leq
\tilde{J}_{\mathfrak{E}}(\tilde{u}_2). 
\]
Indeed, for
	$
	|J_i(u)-J_i(\tilde{u}_1)| \leq |J_i(u)-J_i(\tilde{u}_2)|,
	$
and due to the construction of the error weighting function $\mathfrak{E}$ 
(strictly monotonically increasing in each component), 
	we do not obtain any error  cancellation.
However, since $J_i(u)$ is unknown,
we work with the finer discrete solution $u_h^{(2)}$ 
rather than the exact solution $u$, 
and show that 
	\[
	|J_i(u_h^{(2)})-J_i(\tilde{u}_1)| \leq |J_i(u_h^{(2)})-J_i(\tilde{u}_2)|, 
	\]
holds true. In other words,
\[
|J_i(u)-J_i(\tilde{u}_1)| \leq |J_i(u)-J_i(\tilde{u}_2)|
\]
and
\[
J_i(u_h^{(2)}) ~\not \in ~[J_i(\tilde{u}_1),J_i(\tilde{u}_2)]
~\cup~[J_i(\tilde{u}_2),J_i(\tilde{u}_1)]
\]
imply
\[
|J_i(u_h^{(2)})-J_i(\tilde{u}_1)| \leq |J_i(u_h^{(2)})-J_i(\tilde{u}_2)|.
\]
Without loss of generality, we assume that 
$
J_i(u_h^{(2)}) < J_i(\tilde{u}_1)$ and  $J_i(u_h^{(2)}) < J_i(\tilde{u}_2).
$
From Assumption~\ref{Assumption: Better approximation} and 
$J_i(u_h^{(2)}) ~\not \in ~[J_i(\tilde{u}_1),J_i(\tilde{u}_2)]~\cup~[J_i(\tilde{u}_2),J_i(\tilde{u}_1)]$,
we conclude that $J_i(u)$ does not belong to the union of the 
intervals $[J_i(\tilde{u}_1),J_i(\tilde{u}_2)]$ and $[J_i(\tilde{u}_2),J_i(\tilde{u}_1)]$.
We now distinguish two  cases. First, 
if $J_i(\tilde{u}_1)=J_i(\tilde{u}_2)$,  the statement 
\[
|J_i(u_h^{(2)})-J_i(\tilde{u}_1)| \leq |J_i(u_h^{(2)})-J_i(\tilde{u}_2)|
\]
follows immediately.
In the second case, for $J_i(\tilde{u}_1)\not=J_i(\tilde{u}_2)$,  
Assumption~\ref{Assumption: Better approximation} allows us to conclude 
that we  have either
\[
J_i(u_h^{(2)}) \leq J_i(u)  < J_i(\tilde{u}_1) < J_i(\tilde{u}_2)
\]
or  
\[
J_i(u) <J_i(u_h^{(2)})   < J_i(\tilde{u}_1) < J_i(\tilde{u}_2). 
\]
Both cases imply
$
|J_i(u_h^{(2)})-J_i(\tilde{u}_1)| \leq |J_i(u_h^{(2)})-J_i(\tilde{u}_2)|, 
$
which concludes the proof.
\end{proof}

\begin{remark}			
		We notice that in \cite{HaHou03,Ha08,EnWi17}, the functionals were combined as follows
		\begin{equation}\label{J_c}
		{J}_c(v):=\sum_{i=1}^{N}{\frac{\omega_i\text{ sign}(J_i(u_h^{(2)})-J_i(\tilde{u}))}{|J_i(\tilde{u})|}}J_i(v) \quad \forall v\in \bigcap_{i=0}^N \mathcal{D}(J_i).\nonumber
		\end{equation}
		For the error weighting function
                                        $\mathfrak{E}(x,\vec{J}(\tilde{u})):=
                                        \sum_{i=1}^{N} \frac{\omega_i
                                          x_i}{|J_i(\tilde{u})|}$,
        which yields that the error functional $J_\mathfrak{E}$   
                                        coincides with $(-J_c)$ up to a
                                        constant \cite{EnLaWi18},
                                         the condition
\[
J_i(u_h^{(2)}) ~\not \in ~[J_i(\tilde{u}_1),J_i(\tilde{u}_2)]
~\cup~[J_i(\tilde{u}_2),J_i(\tilde{u}_1)]
\]
is not required to avoid error cancellation.
				
\end{remark}

\section{Algorithms}
\label{Section: Algorithms}
		In this section, we describe  
			the algorithmic realizations of
                our theoretical work.
                The spatial discretization is based
                on the finite element method.
		However, the algorithms
		 	presented below can be adapted to
                other discretization techniques as well.
		We use the same finite element 
		discretizations 
                as 
		in our previous work \cite{EnLaWi18}, 
i.e continuous bilinear elements for $U_h$ and $V_h$ and continuous
bi-quadratic elements for the enriched spaces $U_h^{(2)}$ and $ V_h^{(2)}$ in
the two dimensional case.

		\subsection{Newton's algorithm} 
		Newton's method for solving the nonlinear variational 
		problem (\ref{Equation: Discret Primal Problem})
		on  refinement level $l$ 
		is stated in Algorithm~\ref{Algorithm: adaptive_newton}.
		Below we identify $u_h^{l,k}$ with the corresponding vector
		with respect to the chosen basis when we compute 
		$\Vert \delta u_h^{l,k} \Vert_{\ell_\infty}$.  Furthermore for
                the following algorithm let $\varsigma^{l,k}_h$ be defined as 
\[
\varsigma^{l,k}_h :=~\frac{\Vert\delta u^{l,k-1}_h\Vert_{\ell_\infty}}{1-
  (\Vert\delta u^{l,k-1}_h\Vert_{\ell_\infty} /\Vert \delta u^{l,k-2}_h
  \Vert_{\ell_\infty})^2 },
\]
leading to a stopping criteria which is motivated by \cite{Deuflhard2011}.

		\begin{algorithm}[H]
			\caption{Adaptive Newton algorithm for multiple goal functionals  on level $l$}
			\label{Algorithm: adaptive_newton}
			\begin{algorithmic}[1]
				\State 
				Start with some initial guess $u^{l,0}_h \in
				U_h^l$, set $k=0$,
				and set $TOL_{Newton}^l > 0$.
				\While{$\varsigma^{l,k}_h> TOL_{Newton}^l  (\Vert u^{l,k}_h \Vert_{\ell_\infty} + \Vert\delta u^{l,k-1}_h \Vert_{\ell_\infty})$  or  $\varsigma^{l,k}_h<0$}
				\State Solve for $\delta u^{l,k}_h$, 
				$$ \mathcal{A}'(u^{l,k}_h)(\delta u^{l,k}_h,v_h)=-\mathcal{A}(u^{l,k}_h)(v_h)  \quad \forall v_h \in 
				V_h^l.$$
				\State Update : $u^{l,k+1}_h=u^{l,k}_h+\alpha \delta u^{l,k}_h$ for some good choice $\alpha \in (0,1]$.
				\State $k=k+1.$
				\EndWhile
			\end{algorithmic}
		\end{algorithm}

\begin{remark}		
The arising linear systems are solved using the direct solver UMFPACK \cite{UMFPACK}.
\end{remark}

		\begin{remark}
			In Algorithm~\ref{Algorithm: adaptive_newton},
			we choose $\Vert \delta u^{l,-2}_h
                        \Vert_{\ell_\infty}:=1$, $\Vert \delta u^{l,-1}_h
                        \Vert_{\ell_\infty}:=0.99$ and $TOL_{Newton}^l = 10
                        ^{-8} $. 
To compute $\alpha$,
we used the same line search method as described in \cite{EnLaWi18}.
		\end{remark}

		
\subsection{Adaptive Newton algorithms for multiple goal functionals }
		
		In this section, we describe the key algorithm. The basic structure 
		of the algorithm is similar to 
		that
		presented in \cite{EnLaWi18,RanVi2013} and \cite{ErnVohral2013}.
		In contrast to previous work, we replace the stopping criteria 	
$
|A(u^{l,k}_h)(z^{l,k}_h)|> 10^{-2} \eta_h^{l-1},
$
which was used in \cite{EnLaWi18}, 
by 
$
|(J_{\mathfrak{E}}^{(k)})'(\delta u^{l,k}_h)|> 10^{-2} \eta_h^{l-1}.
$
However, this is only possible since we assume that the linear problem is
solved exactly, and we replace the error estimator on the current level 
by that one of the previous level.
%
\begin{remark}
In the algorithms
    developed in  \cite{RanVi2013}, the computation of the adjoint solution could not be avoided since it was also needed to compute the current discretization error estimator.
\end{remark}

		\begin{algorithm}[H]
			\caption{Adaptive Newton algorithm for multiple goal
				functionals on level $l$ } \label{inexat_newton_algorithm_for_multiple_goal_functionals}
			\begin{algorithmic}[1]
				\State Start with  some initial guess $u^{l,0}_h \in U_h^l$ and $k=0$.
				\State Construct $(J_{\mathfrak{E}}^{(0)})'$ constructed with $u^{l,(2)}_h$ and $u^{l,0}_h$
				\State For $\delta u^{l,k}_h$, solve $$ \mathcal{A}'(u^{l,k}_h)(\delta u^{l,k}_h,v_h)=-\mathcal{A}(u^{l,k}_h)(v_h)  \quad \forall v_h \in V_h^l.$$
				\While{$|(J_{\mathfrak{E}}^{(k)})'(\delta u^{l,k}_h)|> 10^{-2} \eta_h^{l-1}$}
			
				\State Update : $u^{l,k+1}_h=u^{l,k}_h+\alpha \delta u^{l,k}_h$ for some good choice $\alpha \in (0,1]$.
				\State $k=k+1.$
					\State For $\delta u^{l,k}_h$, solve $$ \mathcal{A}'(u^{l,k}_h)(\delta u^{l,k}_h,v_h)=-\mathcal{A}(u^{l,k}_h)(v_h)  \quad \forall v_h \in V_h^l.$$
			\State Construct $(J_{\mathfrak{E}}^{(k)})'$ constructed with $u^{l,(2)}_h$ and $u^{l,k}_h$
				\EndWhile
			\end{algorithmic}
		\end{algorithm}

\begin{remark}
The last Newton update in Algorithm
\ref{inexat_newton_algorithm_for_multiple_goal_functionals} is only used in
the stopping criterion. Of course, one can use this update to perform a very
last Newton update step for a final improvement of the solution.
\end{remark}
\begin{remark}	
We can also use Algorithm~\ref{inexat_newton_algorithm_for_multiple_goal_functionals}
for the enriched problem,
replacing the stopping criterion
$
|(J_{\mathfrak{E}}^{(k)})'(\delta u^{l,k}_h)|> 10^{-2} \eta_h^{l-1}
$
by
$
|J_i'(\delta u^{l,k}_h)|< TOL_i^l.
$
This stopping criterion can also be used for this algorithm, which makes Algorithm \ref{final algorithm} more flexible.
		\end{remark}

		
\subsection{The final algorithm} 
In this subsection, we formulate the overall algorithm starting with an initial mesh $\mathcal{T}_h^1$ and the corresponding finite element spaces 
		$V_h^1$,  $U_h^1$, $U_{h}^{1,(2)}$ and $V_{h}^{1,(2)}$, where $U_{h}^{1,(2)}$ and $V_{h}^{1,(2)}$ are the
		enriched finite element spaces.
		The refinement procedure creates  a sequence of finer and finer 
		meshes
		$\mathcal{T}_h^l$  leading to the corresponding finite element spaces
		$V_h^l$,  $U_h^l$, $U_{h}^{l,(2)}$ and $V_{h}^{l,(2)}$ for $l=2,3,\ldots$ .

		\begin{algorithm}[H]
			\caption{The final algorithm }\label{final algorithm}
			\begin{algorithmic}[1]
				\State Start with some initial guess 
				$u_{h}^{0,(2)}$,$u_h^{0}$, set $l=1$
				and set $TOL_{dis} > 0$.
				\State Solve  (\ref{Equation: Discret Primal Problem}) for  $u_h^{l,(2)}$ using Algorithm \ref{Algorithm: adaptive_newton} with the initial guess $u_h^{l-1,(2)}$ on the discrete space $U_{h}^{l,(2)}$. \label{solve Uh2}
				\State Solve (\ref{Equation: Discret Primal Problem}) using Algorithm \ref{inexat_newton_algorithm_for_multiple_goal_functionals} with the initial guess $u_h^{l-1}$ on the discrete spaces $U_{h}^l$.  \label{final algorithm: solveprimal}
				\State Construct the combined functional $J_{\mathfrak{E}}$ as in (\ref{ErrorrepresentationFunctionalapprox}).
				\State Solve the adjoint problem (\ref{Equation : discrete adjoint Problem}) for $J_{\mathfrak{E}}$ on $V_h^{l,(2)}$ and $V_h^l$. \label{final algorithm: solveadjoint}
				\State Construct the error estimator $\eta_K$  by distributing $\eta_i$ defined in (\ref{eta_i_PU}) to the elements and adding the local remainder contributions $\eta_{\mathcal{R},K}^{(2)} $ defined in (\ref{Equation: LocalErrorContributionRemainder}).
				\State Mark elements with some refinement strategy. \label{final algorithm: refinement}
				\State Refine marked elements: $\mathcal{T}_h^l \mapsto\mathcal{T }_h^{l+1}$ and $l=l+1$.
				\State If $|\eta_h| < TOL_{dis}$ stop, else go to \ref{solve Uh2}.
			\end{algorithmic}
		\end{algorithm}
		
		As already explained above, we replace  the estimated error
                $\eta_h^{l,(2)}$ by  $\eta_h^{l-1,(2)}$ to avoid the
                evaluation of the error estimator and the computation of the
                adjoint solution in step~\ref{final algorithm: solveprimal} 
		of  Algorithm~ \ref{inexat_newton_algorithm_for_multiple_goal_functionals}.
		Thus,  $\eta_h^{l-1}$ is not defined on the first 
		level.
		Therefore, we set 
		$\eta_h^{0}:=10^{-8}$. 
		This means that we perform more iterations on the coarsest level. 
		However, solving on this level is very cheap.

\begin{remark}
The refinement procedure used in our numerical examples  in step \ref{final algorithm: refinement} of Algorithm \ref{final algorithm} is based on the \textit{fixed-rate strategy} described in \cite{BaRa03} with $X = 0.1$ and $Y=0.0$. 
However, in contrast to this procedure, we additionally mark one more element and all elements with the same error contribution as the smallest of the marked cells.
\end{remark}

\section{Numerical examples}
\label{Section: Numerical examples}
                In order to support our theoretical and algorithmic
                developments, some numerical tests are performed in this section.
	        These examples are based on a regularized $p$-Laplace equation  with a very small regularization parameter $\varepsilon >0$ and  $p \in (1, \infty)$. 
	        
		In the first example, we consider a problem with a non-smooth
                analytical solution on the unit square. Here we investigate the behavior in the case of 
                single goal
                functionals. In the second example, we investigate the
                behavior of multiple goal functionals on a more complicated domain. 
                In these tests, we also provide computational hints on the
                validity of the saturation assumptions (despite that for the
                specific choices, we cannot proof that the saturation
                assumptions hold true).
                The implementation is based on the finite element library 
		deal.II \cite{dealII84} and the extension of our previous work \cite{EnWi17}.

\subsection{A single goal functional}
In the first set of computations, we consider the boundary value problem
\begin{equation}
\label{Equation: Example1}
 -\text{div}((\varepsilon^2 + |{\nabla u}|^2)^{\frac{p-2}{2}}\nabla u) =f \;   \text{in }\Omega
 \quad \text{and}  \quad 
				u = 0 \; \text{on } \partial \Omega,
\end{equation}
		with $p=4$ and $\varepsilon = 10^{-10}$ and $f$ such that $u(x,y)=\sqrt{x^2+y^2}(x^2-1)(y^2-1)$ is the exact solution. The computational domain $\Omega$ is the unit square $(-1,1) \times (-1,1)$.
		As goal functional we consider 
\[
J (u):= u(0,0)=0.
\]
This point is exactly where the singularity of the solution is, which is
visualized on Figure~\ref{figure: solution+mesh+adjointsolution}
(left). Furthermore, there is a line singularity, where $\nabla u =0$ leading to additional refinement on these lines and a high gradient of the adjoint solution $z_h$ due to the small regularization parameter, which can be monitored at Figure~\ref{figure: solution+mesh+adjointsolution} (middle, right).
		\begin{figure}[H]
				\includegraphics[width=0.333\textwidth]{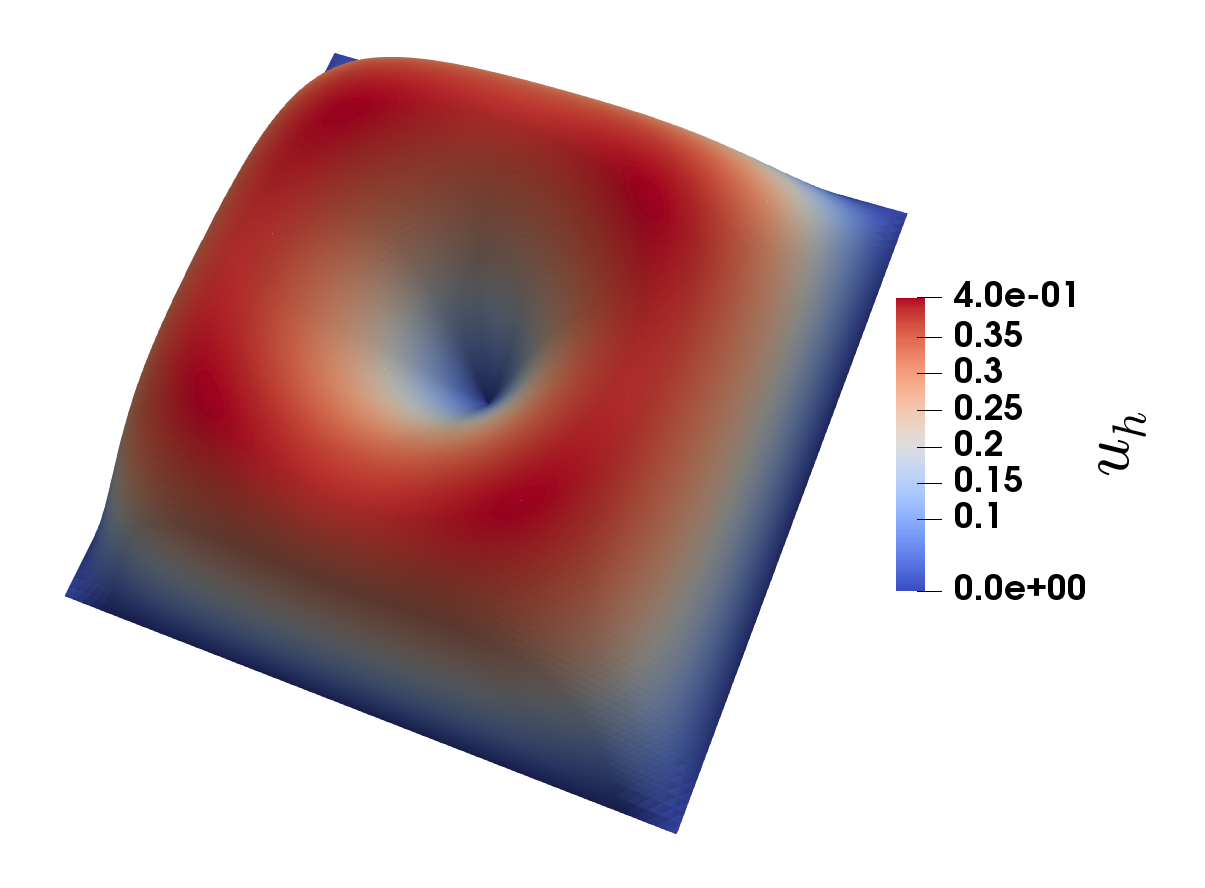}			
				\includegraphics[width=0.333\textwidth]{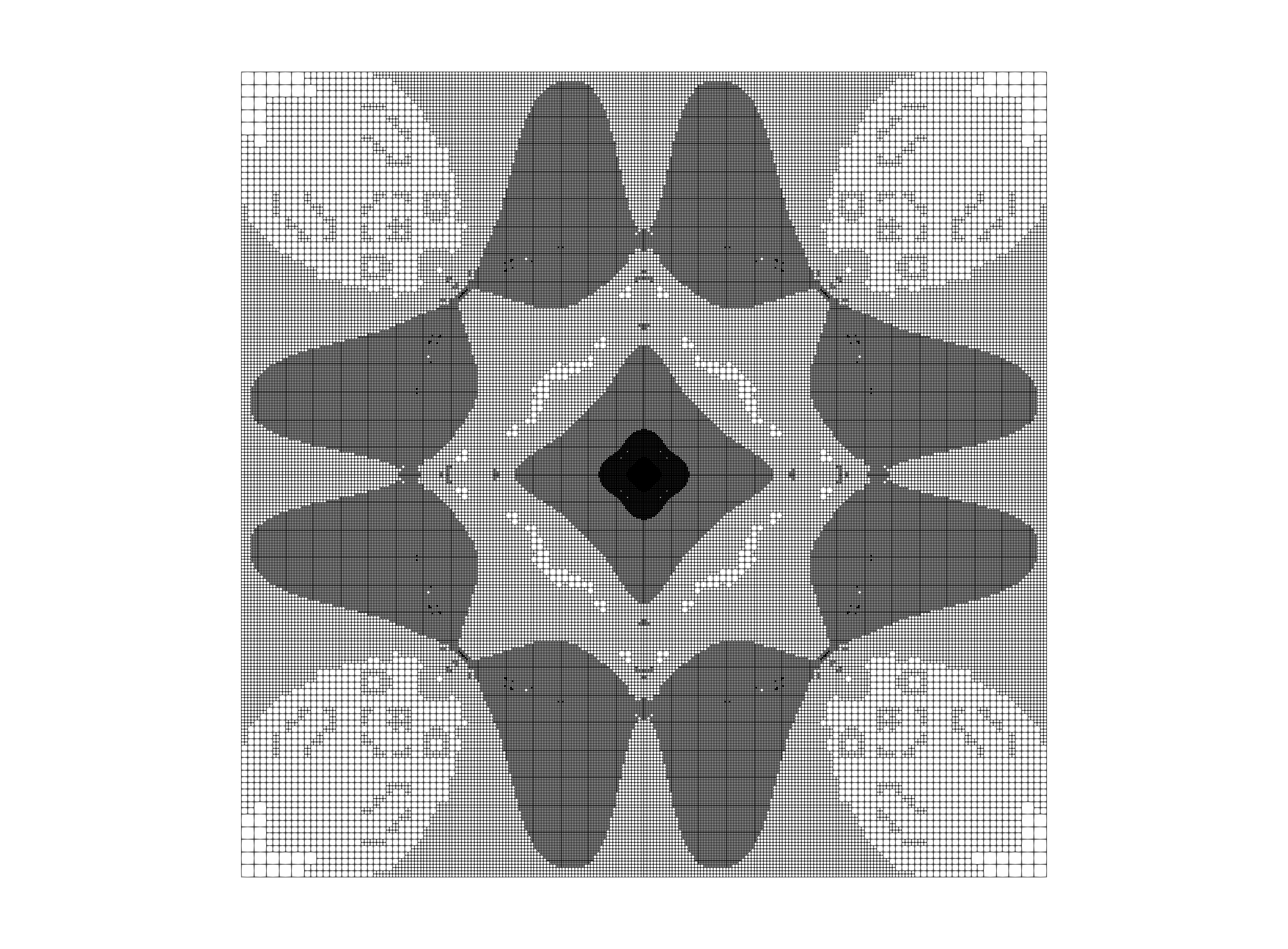}
				\includegraphics[width=0.333\textwidth]{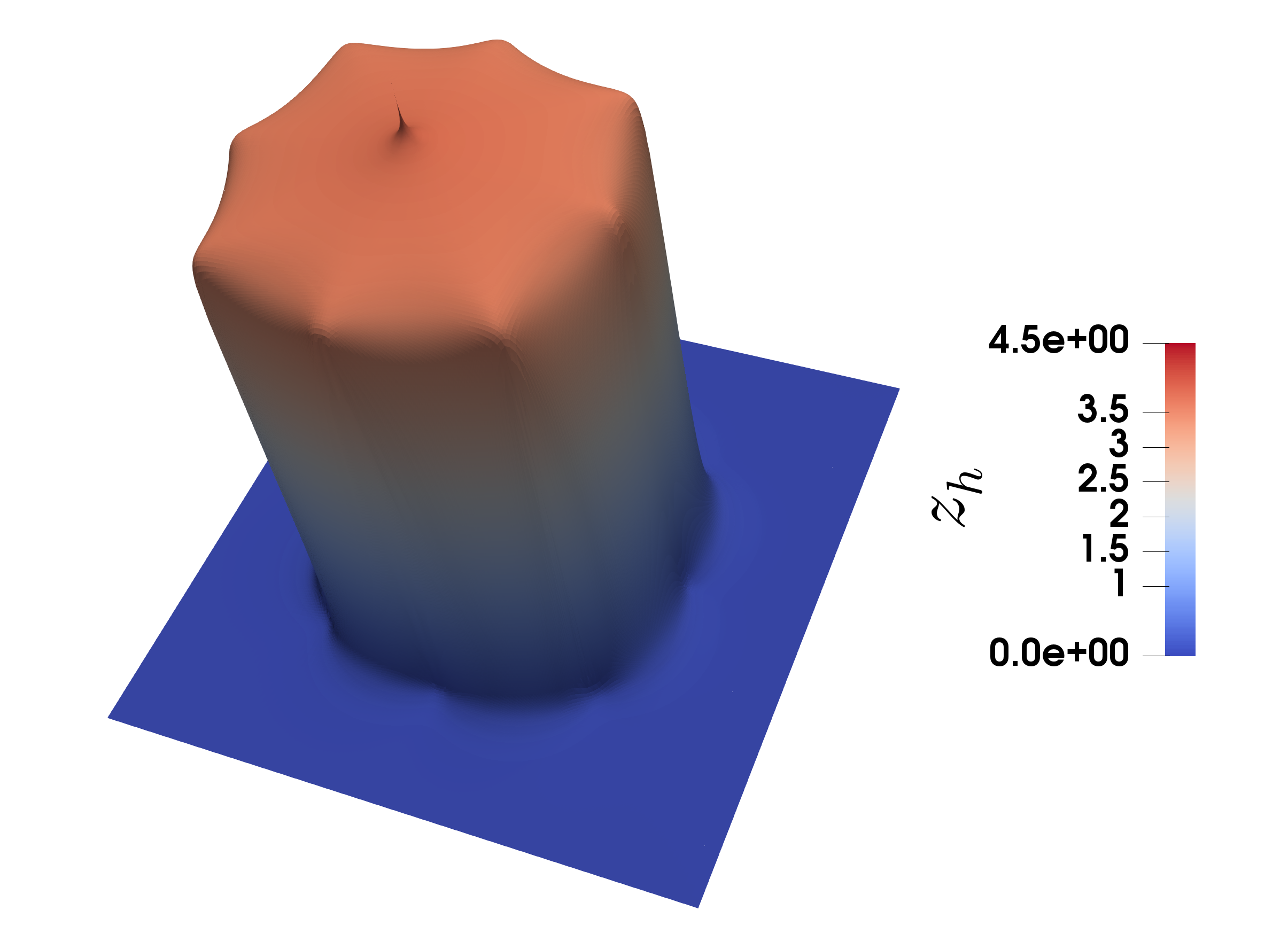}				
				\caption{Approximation of the solution (left), the adjoint solution (right)
					 on the mesh (middle) achieved on level $l$=31 ($144$ $785$ DOFs).} \label{figure: solution+mesh+adjointsolution}
		\end{figure}%

			Inspecting the error in our goal functional $u(0,0)$ for uniform refinement shown in Figure~\ref{figure: Example1Errors}, it turns out that we have a worser convergence rate than $\mathcal{O}(\text{DOFs}^{-\frac{1}{2}}) \approx \mathcal{O}(h)$.
				Adaptivity leads to
			a convergence rate of approximately $\mathcal{O}(\text{DOFs}^{-1})$, 
			i.e., 
			to reach the same accuracy as 
			with more than $1\, 000\, 000$ DOFs using uniform
                        refinement, we need less than 10 000 DOFs.

			Furthermore, we monitor that the influences of the remainder term and iteration error vanish during the refinement process, as expected. 
			Specifically, the estimator part $|\eta^{(2)}_\mathcal{R}|$ shows a higher-order behavior as expected, 
			but has an influence on coarse meshes. 

		Moreover, in this numerical example, Assumption~\ref{Assumption: Better approximation} and Assumption~\ref{Assumption: With Remainder} seem to be fulfilled even with the additional condition that $b_h \rightarrow 0$ and $b_{h,\gamma} \rightarrow 0$ on adaptive meshes, which we also observe in Figure~\ref{figure: Example1adaptiveIeff} as well as in Table~\ref{Table: IeffAndErrorsSingleGoal}. 
			On the other hand, we observe a completely different behavior on uniformly refined meshes in Figure~\ref{figure: Example1uniformIeff}. The effectivity indices are approximately $0.1 - 0.2$, 
			i.e.,
			our estimator determines the error better on adaptively refined meshes. 
			In  Theorem \ref{Theorem: Ieffbounds}, we prove that the efficiency depends on the constant $b_0$ in the saturation assumption. 
			We assume that, for this example, $b_0$ is closer to $1$ in the case of uniform refinement, 
			while, for adaptive refinement, we also recover parts of the optimal convergence rate for the enriched space, and, therefore, we obtain $b_h \rightarrow 0$ and $b_{h,\gamma} \rightarrow 0$.
	
		\begin{table}[]
			\centering
			\begin{tabular}{|l|r|l|l|c|c|c|}
				\hline
				$l$ & {DOFs} & $I_{eff} $& $I_{eff,\gamma} $ & $|\eta^{(2)}|$ & $|\eta_h^{(2)}|$ & $|J(u)-J(u_h)|$ \\ \hline
				1              & 9                                 & 0.753    & 0.237 & 7.17E-01               & 2.26E-01         & 9.52E-01     \\ \hline
				2              & 25                                & 1.007    & 1.836 & 1.93E-01               & 3.51E-01         & 1.91E-01     \\ \hline
				5              & 133                               & 0.608    & 0.910 & 4.80E-02               & 7.18E-02         & 7.89E-02     \\ \hline
				10             & 605                               & 0.745    & 0.858 & 1.04E-02               & 1.20E-02         & 1.39E-02     \\ \hline
				15             & 2 365                              & 0.882    & 0.877 & 2.61E-03               & 2.59E-03         & 2.95E-03     \\ \hline
				20             & 8 481                              & 0.923    & 0.917 & 6.00E-04               & 5.95E-04         & 6.49E-04     \\ \hline
				25             & 31 649                             & 0.984    & 0.973 & 2.00E-04               & 1.98E-04         & 2.03E-04     \\ \hline
				30             & 111 793                            & 0.995    & 0.992 & 4.27E-05               & 4.26E-05         & 4.29E-05     \\ \hline
				35             & 410 201                            & 0.999    & 0.996 & 1.12E-05               & 1.12E-05         & 1.12E-05     \\ \hline
				39             & 1 166 237                           & 1.000    & 1.004 & 3.74E-06               & 3.76E-06         & 3.74E-06     \\ \hline
				40             & 1 513 865                           & 1.000    & 1.000 & 2.96E-06               & 2.96E-06         & 2.96E-06     \\ \hline
			\end{tabular}
			\caption{Effectivity indices and errors for adaptive
                          refinement. Here, $l$ denotes the refinement
                          level. Several intermediate levels are left out for
                          the sake of a clearly arranged table.}\label{Table: IeffAndErrorsSingleGoal}
		\end{table}

\begin{figure}[h]
	\centering			
	\scalebox{1.0}{\input{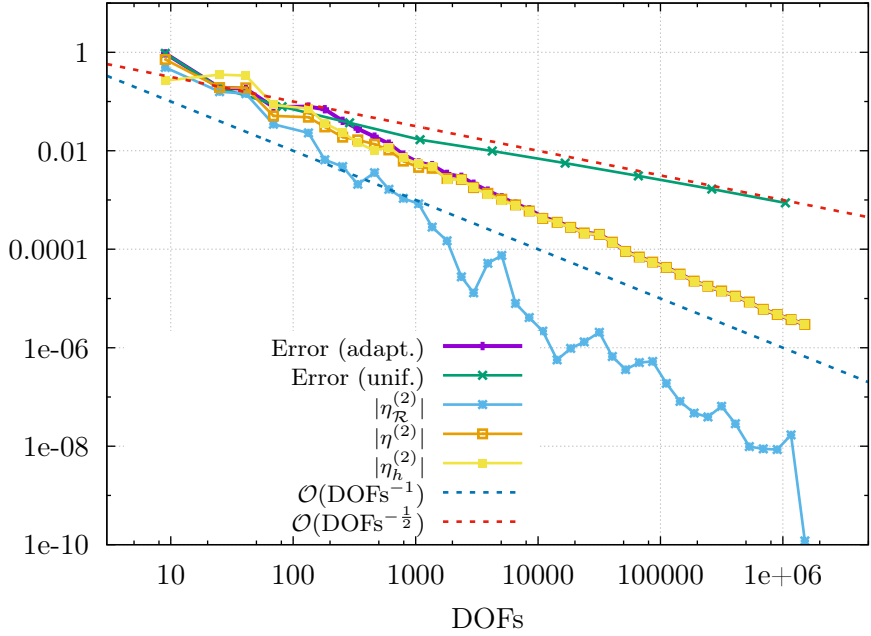}}
	\captionof{figure}{ Error vs DOFs for $p=4$, $\varepsilon=10^{-10}$. Error (unif.) describes the error for uniform refinement in $J(u)$ and Error (adapt.) for adaptive refinement.}\label{figure: Example1Errors}
\end{figure}
		
		\begin{minipage}[t]{0.45 \textwidth}
				\scalebox{0.65}{
\begingroup
  \makeatletter
  \providecommand\color[2][]{%
    \GenericError{(gnuplot) \space\space\space\@spaces}{%
      Package color not loaded in conjunction with
      terminal option `colourtext'%
    }{See the gnuplot documentation for explanation.%
    }{Either use 'blacktext' in gnuplot or load the package
      color.sty in LaTeX.}%
    \renewcommand\color[2][]{}%
  }%
  \providecommand\includegraphics[2][]{%
    \GenericError{(gnuplot) \space\space\space\@spaces}{%
      Package graphicx or graphics not loaded%
    }{See the gnuplot documentation for explanation.%
    }{The gnuplot epslatex terminal needs graphicx.sty or graphics.sty.}%
    \renewcommand\includegraphics[2][]{}%
  }%
  \providecommand\rotatebox[2]{#2}%
  \@ifundefined{ifGPcolor}{%
    \newif\ifGPcolor
    \GPcolorfalse
  }{}%
  \@ifundefined{ifGPblacktext}{%
    \newif\ifGPblacktext
    \GPblacktexttrue
  }{}%
  \let\gplgaddtomacro\g@addto@macro
  \gdef\gplbacktext{}%
  \gdef\gplfronttext{}%
  \makeatother
  \ifGPblacktext
    \def\colorrgb#1{}%
    \def\colorgray#1{}%
  \else
    \ifGPcolor
      \def\colorrgb#1{\color[rgb]{#1}}%
      \def\colorgray#1{\color[gray]{#1}}%
      \expandafter\def\csname LTw\endcsname{\color{white}}%
      \expandafter\def\csname LTb\endcsname{\color{black}}%
      \expandafter\def\csname LTa\endcsname{\color{black}}%
      \expandafter\def\csname LT0\endcsname{\color[rgb]{1,0,0}}%
      \expandafter\def\csname LT1\endcsname{\color[rgb]{0,1,0}}%
      \expandafter\def\csname LT2\endcsname{\color[rgb]{0,0,1}}%
      \expandafter\def\csname LT3\endcsname{\color[rgb]{1,0,1}}%
      \expandafter\def\csname LT4\endcsname{\color[rgb]{0,1,1}}%
      \expandafter\def\csname LT5\endcsname{\color[rgb]{1,1,0}}%
      \expandafter\def\csname LT6\endcsname{\color[rgb]{0,0,0}}%
      \expandafter\def\csname LT7\endcsname{\color[rgb]{1,0.3,0}}%
      \expandafter\def\csname LT8\endcsname{\color[rgb]{0.5,0.5,0.5}}%
    \else
      \def\colorrgb#1{\color{black}}%
      \def\colorgray#1{\color[gray]{#1}}%
      \expandafter\def\csname LTw\endcsname{\color{white}}%
      \expandafter\def\csname LTb\endcsname{\color{black}}%
      \expandafter\def\csname LTa\endcsname{\color{black}}%
      \expandafter\def\csname LT0\endcsname{\color{black}}%
      \expandafter\def\csname LT1\endcsname{\color{black}}%
      \expandafter\def\csname LT2\endcsname{\color{black}}%
      \expandafter\def\csname LT3\endcsname{\color{black}}%
      \expandafter\def\csname LT4\endcsname{\color{black}}%
      \expandafter\def\csname LT5\endcsname{\color{black}}%
      \expandafter\def\csname LT6\endcsname{\color{black}}%
      \expandafter\def\csname LT7\endcsname{\color{black}}%
      \expandafter\def\csname LT8\endcsname{\color{black}}%
    \fi
  \fi
    \setlength{\unitlength}{0.0500bp}%
    \ifx\gptboxheight\undefined%
      \newlength{\gptboxheight}%
      \newlength{\gptboxwidth}%
      \newsavebox{\gptboxtext}%
    \fi%
    \setlength{\fboxrule}{0.5pt}%
    \setlength{\fboxsep}{1pt}%
\begin{picture}(7200.00,5040.00)%
    \gplgaddtomacro\gplbacktext{%
      \csname LTb\endcsname%
      \put(814,704){\makebox(0,0)[r]{\strut{}0}}%
      \csname LTb\endcsname%
      \put(814,1518){\makebox(0,0)[r]{\strut{}0.5}}%
      \csname LTb\endcsname%
      \put(814,2332){\makebox(0,0)[r]{\strut{}1}}%
      \csname LTb\endcsname%
      \put(814,3147){\makebox(0,0)[r]{\strut{}1.5}}%
      \csname LTb\endcsname%
      \put(814,3961){\makebox(0,0)[r]{\strut{}2}}%
      \csname LTb\endcsname%
      \put(814,4775){\makebox(0,0)[r]{\strut{}2.5}}%
      \csname LTb\endcsname%
      \put(946,484){\makebox(0,0){\strut{}1}}%
      \csname LTb\endcsname%
      \put(1783,484){\makebox(0,0){\strut{}10}}%
      \csname LTb\endcsname%
      \put(2619,484){\makebox(0,0){\strut{}100}}%
      \csname LTb\endcsname%
      \put(3456,484){\makebox(0,0){\strut{}1000}}%
      \csname LTb\endcsname%
      \put(4293,484){\makebox(0,0){\strut{}10000}}%
      \csname LTb\endcsname%
      \put(5130,484){\makebox(0,0){\strut{}100000}}%
      \csname LTb\endcsname%
      \put(5966,484){\makebox(0,0){\strut{}1e+06}}%
      \csname LTb\endcsname%
      \put(6803,484){\makebox(0,0){\strut{}1e+07}}%
    }%
    \gplgaddtomacro\gplfronttext{%
      \csname LTb\endcsname%
      \put(176,2739){\rotatebox{-270}{\makebox(0,0){\strut{}effectivity indices}}}%
      \put(3874,154){\makebox(0,0){\strut{}\text{DOFs}}}%
      \csname LTb\endcsname%
      \put(5816,4602){\makebox(0,0)[r]{\strut{}$I_{eff,\gamma}$}}%
      \csname LTb\endcsname%
      \put(5816,4382){\makebox(0,0)[r]{\strut{}$I_{eff}$}}%
      \csname LTb\endcsname%
      \put(5816,4162){\makebox(0,0)[r]{\strut{}$I_{eff,p}$}}%
      \csname LTb\endcsname%
      \put(5816,3942){\makebox(0,0)[r]{\strut{}$I_{eff,a}$}}%
      \csname LTb\endcsname%
      \put(5816,3722){\makebox(0,0)[r]{\strut{}1}}%
    }%
    \gplbacktext
    \put(0,0){\includegraphics{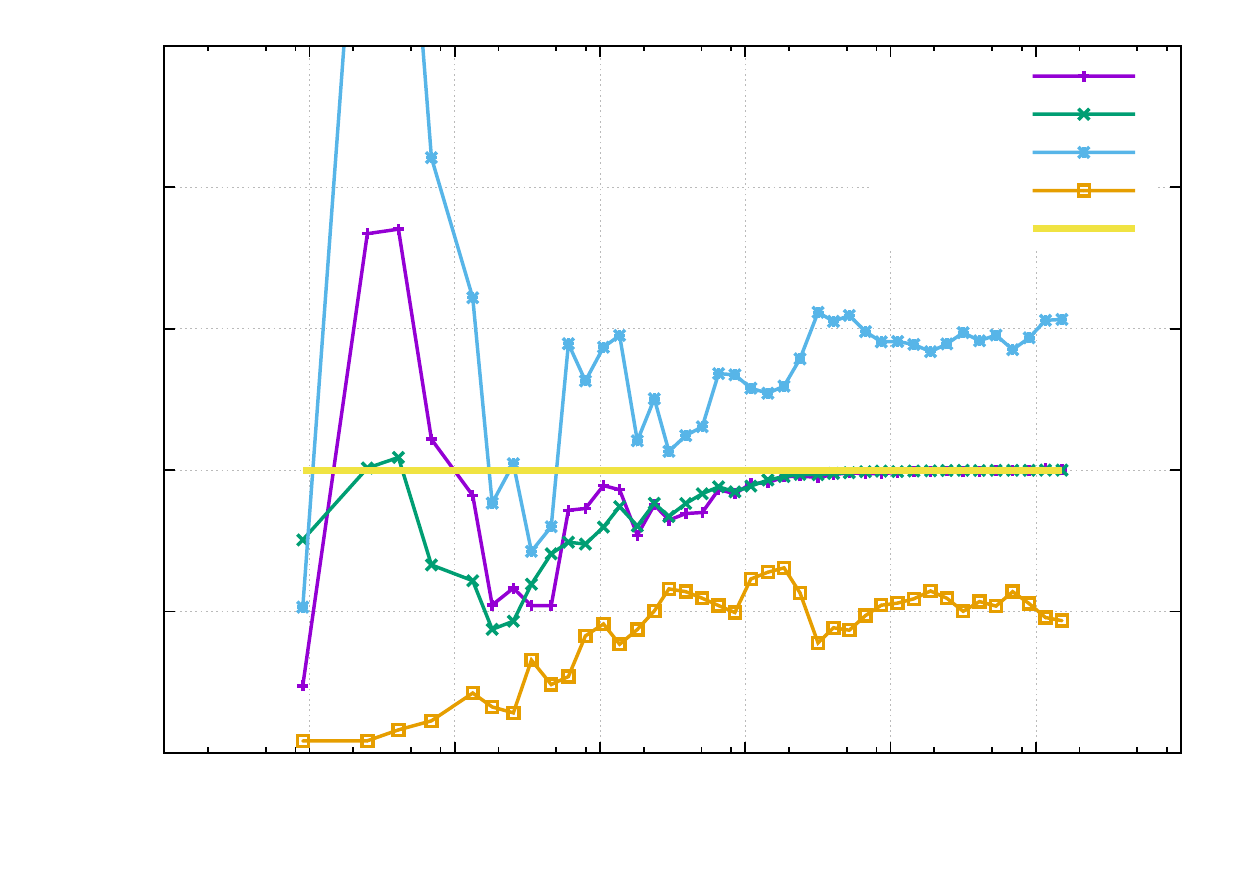}}%
    \gplfronttext
  \end{picture}%
\endgroup
}
				\captionof{figure}{Effectivity indices for adaptive refinement}\label{figure: Example1adaptiveIeff}			
		\end{minipage} %
		\hfill
			\begin{minipage}[t]{0.45 \textwidth}
			\scalebox{0.65}{
\begingroup
  \makeatletter
  \providecommand\color[2][]{%
    \GenericError{(gnuplot) \space\space\space\@spaces}{%
      Package color not loaded in conjunction with
      terminal option `colourtext'%
    }{See the gnuplot documentation for explanation.%
    }{Either use 'blacktext' in gnuplot or load the package
      color.sty in LaTeX.}%
    \renewcommand\color[2][]{}%
  }%
  \providecommand\includegraphics[2][]{%
    \GenericError{(gnuplot) \space\space\space\@spaces}{%
      Package graphicx or graphics not loaded%
    }{See the gnuplot documentation for explanation.%
    }{The gnuplot epslatex terminal needs graphicx.sty or graphics.sty.}%
    \renewcommand\includegraphics[2][]{}%
  }%
  \providecommand\rotatebox[2]{#2}%
  \@ifundefined{ifGPcolor}{%
    \newif\ifGPcolor
    \GPcolorfalse
  }{}%
  \@ifundefined{ifGPblacktext}{%
    \newif\ifGPblacktext
    \GPblacktexttrue
  }{}%
  \let\gplgaddtomacro\g@addto@macro
  \gdef\gplbacktext{}%
  \gdef\gplfronttext{}%
  \makeatother
  \ifGPblacktext
    \def\colorrgb#1{}%
    \def\colorgray#1{}%
  \else
    \ifGPcolor
      \def\colorrgb#1{\color[rgb]{#1}}%
      \def\colorgray#1{\color[gray]{#1}}%
      \expandafter\def\csname LTw\endcsname{\color{white}}%
      \expandafter\def\csname LTb\endcsname{\color{black}}%
      \expandafter\def\csname LTa\endcsname{\color{black}}%
      \expandafter\def\csname LT0\endcsname{\color[rgb]{1,0,0}}%
      \expandafter\def\csname LT1\endcsname{\color[rgb]{0,1,0}}%
      \expandafter\def\csname LT2\endcsname{\color[rgb]{0,0,1}}%
      \expandafter\def\csname LT3\endcsname{\color[rgb]{1,0,1}}%
      \expandafter\def\csname LT4\endcsname{\color[rgb]{0,1,1}}%
      \expandafter\def\csname LT5\endcsname{\color[rgb]{1,1,0}}%
      \expandafter\def\csname LT6\endcsname{\color[rgb]{0,0,0}}%
      \expandafter\def\csname LT7\endcsname{\color[rgb]{1,0.3,0}}%
      \expandafter\def\csname LT8\endcsname{\color[rgb]{0.5,0.5,0.5}}%
    \else
      \def\colorrgb#1{\color{black}}%
      \def\colorgray#1{\color[gray]{#1}}%
      \expandafter\def\csname LTw\endcsname{\color{white}}%
      \expandafter\def\csname LTb\endcsname{\color{black}}%
      \expandafter\def\csname LTa\endcsname{\color{black}}%
      \expandafter\def\csname LT0\endcsname{\color{black}}%
      \expandafter\def\csname LT1\endcsname{\color{black}}%
      \expandafter\def\csname LT2\endcsname{\color{black}}%
      \expandafter\def\csname LT3\endcsname{\color{black}}%
      \expandafter\def\csname LT4\endcsname{\color{black}}%
      \expandafter\def\csname LT5\endcsname{\color{black}}%
      \expandafter\def\csname LT6\endcsname{\color{black}}%
      \expandafter\def\csname LT7\endcsname{\color{black}}%
      \expandafter\def\csname LT8\endcsname{\color{black}}%
    \fi
  \fi
    \setlength{\unitlength}{0.0500bp}%
    \ifx\gptboxheight\undefined%
      \newlength{\gptboxheight}%
      \newlength{\gptboxwidth}%
      \newsavebox{\gptboxtext}%
    \fi%
    \setlength{\fboxrule}{0.5pt}%
    \setlength{\fboxsep}{1pt}%
\begin{picture}(7200.00,5040.00)%
    \gplgaddtomacro\gplbacktext{%
      \csname LTb\endcsname%
      \put(814,704){\makebox(0,0)[r]{\strut{}0}}%
      \csname LTb\endcsname%
      \put(814,1518){\makebox(0,0)[r]{\strut{}0.5}}%
      \csname LTb\endcsname%
      \put(814,2332){\makebox(0,0)[r]{\strut{}1}}%
      \csname LTb\endcsname%
      \put(814,3147){\makebox(0,0)[r]{\strut{}1.5}}%
      \csname LTb\endcsname%
      \put(814,3961){\makebox(0,0)[r]{\strut{}2}}%
      \csname LTb\endcsname%
      \put(814,4775){\makebox(0,0)[r]{\strut{}2.5}}%
      \csname LTb\endcsname%
      \put(946,484){\makebox(0,0){\strut{}1}}%
      \csname LTb\endcsname%
      \put(1783,484){\makebox(0,0){\strut{}10}}%
      \csname LTb\endcsname%
      \put(2619,484){\makebox(0,0){\strut{}100}}%
      \csname LTb\endcsname%
      \put(3456,484){\makebox(0,0){\strut{}1000}}%
      \csname LTb\endcsname%
      \put(4293,484){\makebox(0,0){\strut{}10000}}%
      \csname LTb\endcsname%
      \put(5130,484){\makebox(0,0){\strut{}100000}}%
      \csname LTb\endcsname%
      \put(5966,484){\makebox(0,0){\strut{}1e+06}}%
      \csname LTb\endcsname%
      \put(6803,484){\makebox(0,0){\strut{}1e+07}}%
    }%
    \gplgaddtomacro\gplfronttext{%
      \csname LTb\endcsname%
      \put(176,2739){\rotatebox{-270}{\makebox(0,0){\strut{}effectivity indices}}}%
      \put(3874,154){\makebox(0,0){\strut{}\text{DOFs}}}%
      \csname LTb\endcsname%
      \put(5816,4602){\makebox(0,0)[r]{\strut{}$I_{eff,\gamma}$}}%
      \csname LTb\endcsname%
      \put(5816,4382){\makebox(0,0)[r]{\strut{}$I_{eff}$}}%
      \csname LTb\endcsname%
      \put(5816,4162){\makebox(0,0)[r]{\strut{}$I_{eff,p}$}}%
      \csname LTb\endcsname%
      \put(5816,3942){\makebox(0,0)[r]{\strut{}$I_{eff,a}$}}%
      \csname LTb\endcsname%
      \put(5816,3722){\makebox(0,0)[r]{\strut{}1}}%
    }%
    \gplbacktext
    \put(0,0){\includegraphics{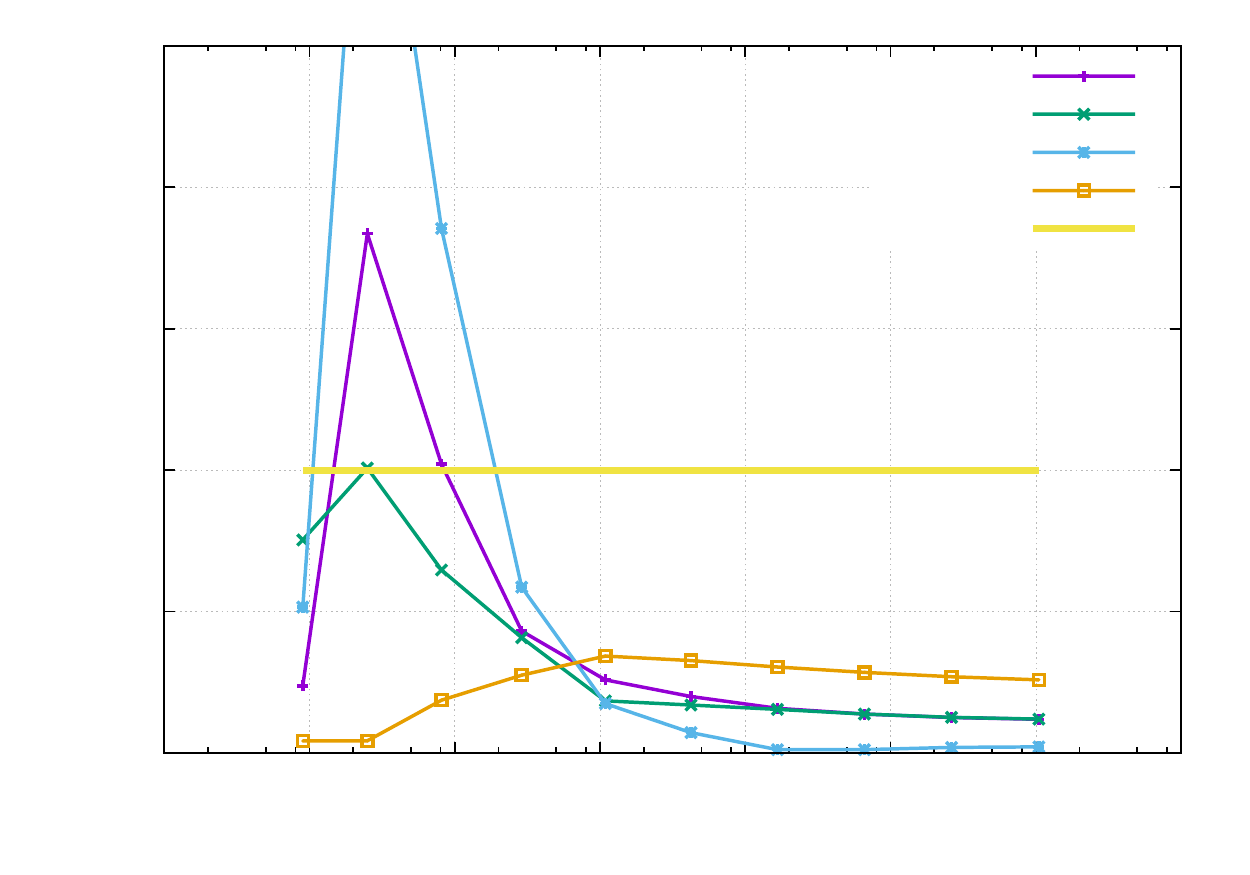}}%
    \gplfronttext
  \end{picture}%
\endgroup
}
			\captionof{figure}{Effectivity indices for uniform refinement}\label{figure: Example1uniformIeff}						
		\end{minipage}
				\begin{figure}[p]
					\includegraphics[width=0.333\textwidth]{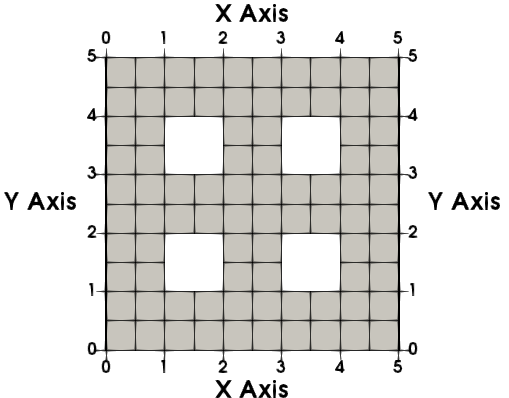}			
					\includegraphics[width=0.333\textwidth]{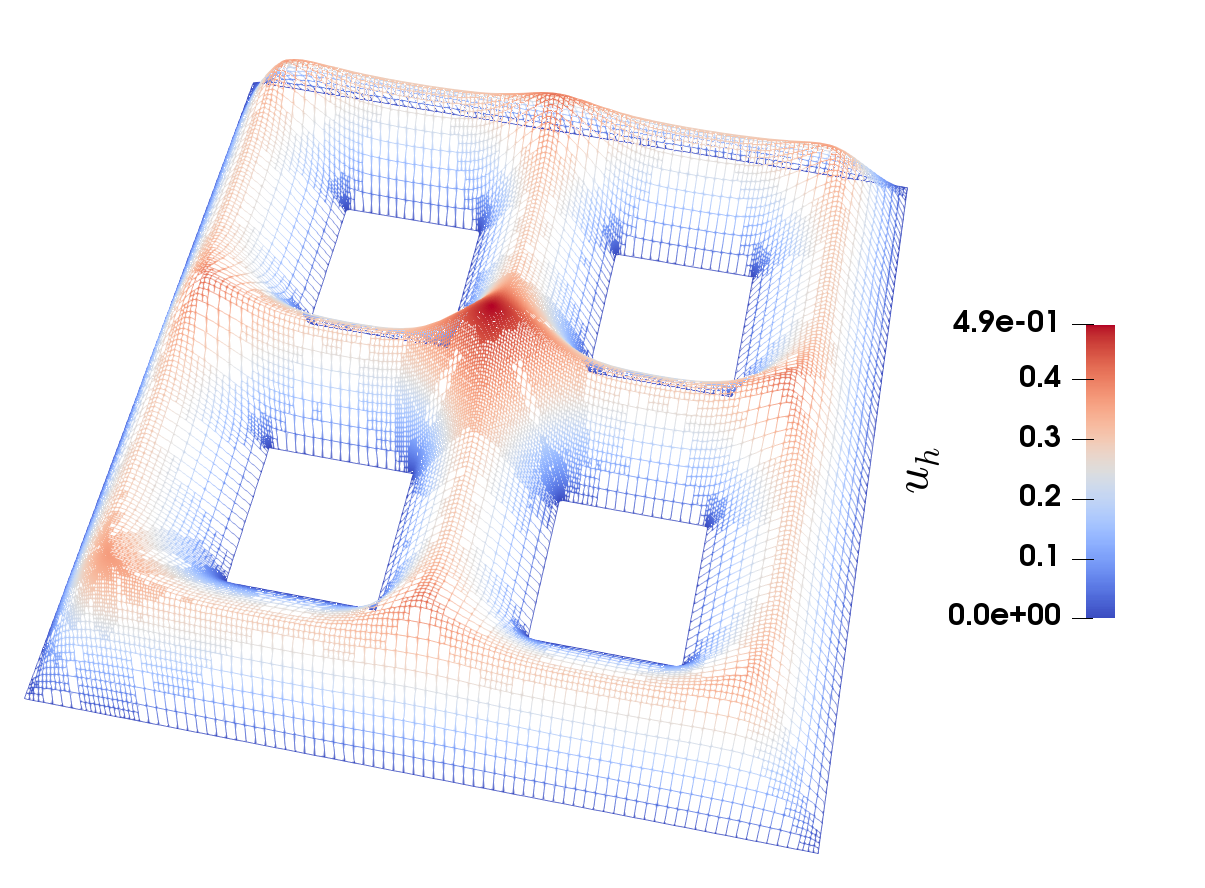}
					\includegraphics[width=0.333\textwidth]{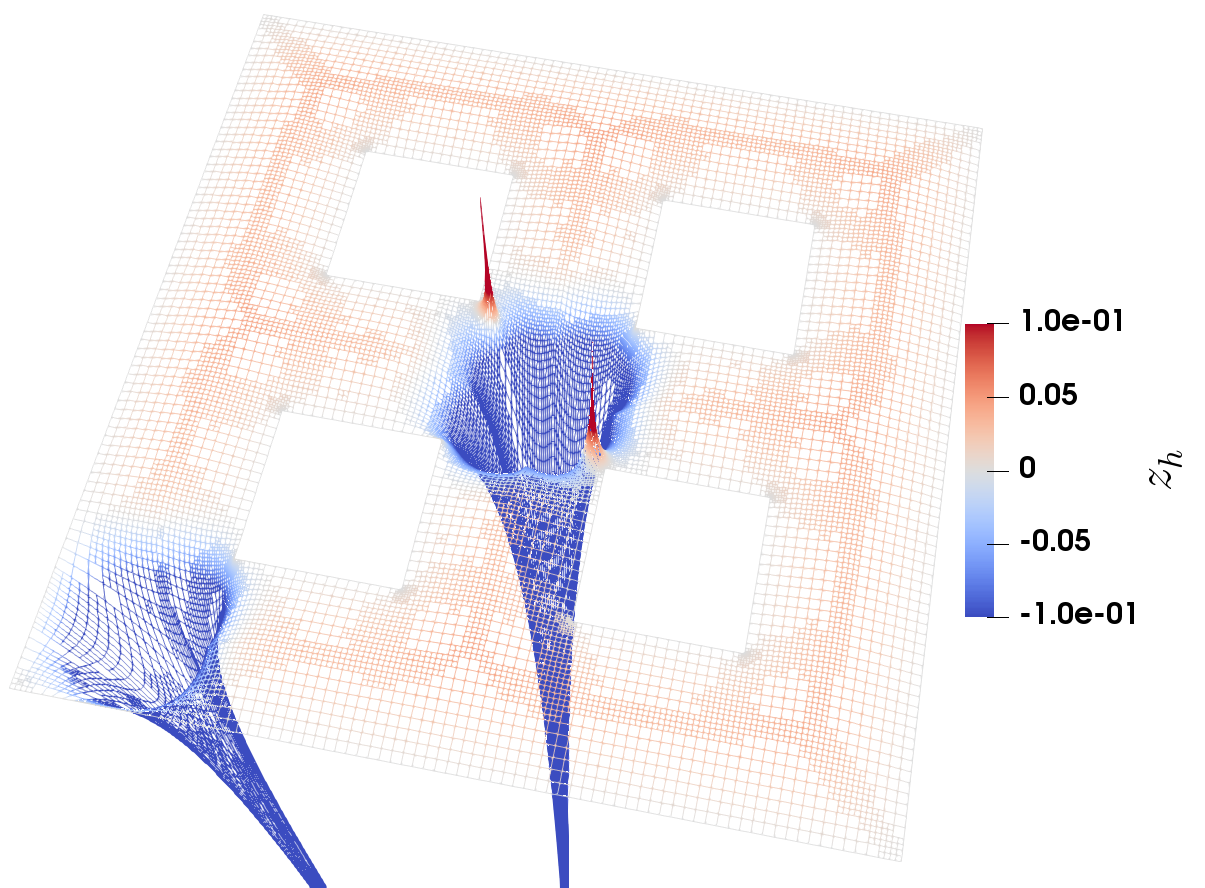}
					\caption{Initial mesh (left), the primal solution (middle) and adjoint solution (right)
						on the mesh achieved on level $l$=22 ($24$ $532$ DOFs).} \label{figure: multigoal + illustration}
				\end{figure}%
				
				\begin{figure}[p]
					\begin{minipage}[t]{0.45 \textwidth}
						\scalebox{0.65}{\input{Figures/Example4_archive.tex}}
						\captionof{figure}{Effectivity indices for adaptive refinement for multiple goals.}\label{Figure: multigoal + Ieffs}
					\end{minipage} %
					\hfill
					\begin{minipage}[t]{0.45 \textwidth}

						\scalebox{0.65}{\input{Figures/Example8_archive.tex}}
						\captionof{figure}{Error reduction for the single functionals using uniform refinement.}\label{Figure: multigoal + uniformerrors}
					\end{minipage}

					\begin{minipage}[t]{0.45 \textwidth}
						\scalebox{0.65}{\input{Figures/Examplewo6_archive.tex}}
						\captionof{figure}{Comparison of the different error parts and the error in the uniform and adaptive case for the combined functional $J_\mathfrak{E}$.}\label{Figure: multigoal + errorestimatorpartscomparison}
					\end{minipage} %
					\hfill
					\begin{minipage}[t]{0.45 \textwidth}
						\scalebox{0.65}{\input{Figures/Examplewo9_archive.tex}}
						\captionof{figure}{Error reduction for the single functionals using adaptive refinement.}\label{Figure: multigoal + adptiveformerrors}
					\end{minipage}
				\end{figure}

\subsection{Multiple goal functionals}

In the second example, 
we again consider 
the non-linear boundary value problem (\ref{Equation: Example1}), but with a different 
right-hand side and a different computational domain $\Omega$. 
Specifically, we choose $f \equiv 1$. 
The computational domain is sketched in Figure~\ref{figure: multigoal + illustration} (left). 
This example already was  considered in our previous work
                        \cite{EnLaWi18} with the same parameters $p = 4$ and
                        $\varepsilon = 10^{-10}$. Furthermore, we consider the
                        same functionals of interest which are given by 
                       \begin{align*}
                       	J_1(u):= &(1+u(2.9,2.1))(1+u(2.1,2.9)),\\
                       	J_2(u):= &\left(\int_{\Omega} u(x,y)-u(2.5,2.5)\,d(x,y)\right)^2,\\	
                       	J_3(u):= &\int_{(2,3)\times(2,3)} u(x,y)\,d(x,y),\\
                       	J_4(u):= &u(0.6,0.6),
                       \end{align*}
                       with the same approximations as in \cite{EnLaWi18},
                       where a reference solution  on a fine grid ($8$ uniform
                       refinements, $Q_c^2$ elements, $22\,038\,525$
                       \text{DOFs}) 
was computed on the cluster RADON1\footnote{\url{https://www.ricam.oeaw.ac.at/hpc/overview/}}.

In the following, we discuss and interpret our observations.
	In Figure \ref{Figure: multigoal + errorestimatorpartscomparison}, we can observe that we indeed obtain an improved convergence rate for our error functional $J_\mathfrak{E}$.
	By comparing the error reductions in the single functionals for uniform and adaptive refinement in  Figure~\ref{Figure: multigoal + uniformerrors} and Figure~\ref{Figure: multigoal + adptiveformerrors}, respectively, we observe similar convergence rates in all functionals as well as an improvement for the adaptive approach. 
	However, this does not necessarily hold true for all functionals as shown in~\cite{EnWi17}.  
	Monitoring the different types of errors, we observe that the remainder part is indeed of higher order.
	Furthermore, both error estimators almost coincide with the true error. 
	This leads to effectivity indices close to one, which are provided in Figure~\ref{Figure: multigoal + Ieffs}. 
	This figure also shows that it is not sufficient to consider only the primal part of the error estimator.

\section{Conclusions}
\label{Section: Conclusions}

In this work, we further investigated and developed a posteriori error
                estimation and mesh adaptivity using the dual-weighted
                residual method for treating multiple goal functionals. This 
                framework includes both nonlinear PDEs and nonlinear goal
                functionals, estimation of the discretization error and the 
                nonlinear iteration error. The latter can be used as stopping 
                criterion for the nonlinear solver,
                	e.g., for the Newton solver that is used in our 
                numerical experiments.
                Using a saturation assumption, we could establish the
                efficiency of the error estimator. These theoretical findings 
                give insight 
                into
                the influence of the choice of the enriched space 
                	that is used 
                to approximate the unknown exact solution in the error
                estimator.
                Our developments
                are substantiated with carefully designed numerical tests. 
                Moreover, our studies also include investigations of the influence of the
                remainder term 
                to
                the error estimator. 
                Summarizing,
                we have designed a well-tested
                framework for the regularized $p$-Laplacian that will be extended in future work to some of the promised 
                (stationary) multiphysics applications mentioned in the introduction.

   \section{Acknowledgments}
   This work has been supported by the Austrian Science Fund (FWF) under the grant
   P 29181
   `Goal-Oriented Error Control for Phase-Field Fracture Coupled to Multiphysics Problems'.
   The third author was supported by RICAM
   during his visit at
   Linz in August 2018.
   The authors would like to thank
   D. Jodlbauer, 
   A. Schafelner and
   W. Zulehner for helpful discussions.

		%
		\bibliography{./lit.lib}
		\bibliographystyle{abbrv}

	\end{document}